\documentclass[a4paper,reqno, 11pt]{amsart}

\usepackage{graphicx,color,marginnote}
\usepackage{amsmath}
\usepackage[utf8]{inputenc}
\usepackage[T1]{fontenc}
\usepackage{cite}
\usepackage[english]{babel}
\usepackage[normalem]{ulem} 

\usepackage{amsfonts}
\usepackage{amssymb}
\usepackage{bbm}
\usepackage{epstopdf}
\usepackage{color}
\numberwithin{equation}{section}

\newcommand{\cT}{\mathcal T}

\newcommand{\cS}{\mathcal S}

\newcommand{\C}{\mathbb{C}}
\newcommand{\R}{\mathbb{R}}
\newcommand{\cO}{\mathcal{O}}

\newcommand{\tb}{\bullet}
\newcommand{\tw}{\circ}
\newcommand{\frb}{\mathfrak{b}}
\newcommand{\frw}{\mathfrak{w}}

\newcommand\ssp{{\scriptscriptstyle +}}
\newcommand\ssm{{\scriptscriptstyle -}}
\newcommand\sspm{{\scriptscriptstyle \pm}}
\newcommand\ssmp{{\scriptscriptstyle \mp}}

\newcommand\FF[1]{F^{[#1]}}
\newcommand\ff[1]{f^{{\scriptscriptstyle [}#1{\scriptscriptstyle ]}}}

\def\spl{\mathrm{spl}}

\def\frw{\mathfrak{w}}
\def\frb{\mathfrak{b}}

\def\fw{{f{}_{\!\frw}}}
\def\fb{{f{}_{\!\frb}}}
\def\ofw{{\overline{f}{}_{\!\frw}}}
\def\ofb{{\overline{f}{}_{\!\frb}}}

\def\pazeta{\partial_{\vphantom{\bar{\zeta}}\zeta}}
\def\opazeta{\partial_{\bar{\zeta}}}

\newcommand{\wc}{u^\tw}
\newcommand{\bc}{u^\tb}

\usepackage{amsthm}
\theoremstyle{plain}
\newtheorem{theorem}{Theorem}[section]

\newtheorem{corollary}[theorem]{Corollary}
\newtheorem{assumption}[theorem]{Assumption}
\newtheorem*{assumption*}{Assumption}

\newtheorem{lemma}[theorem]{Lemma}
\newtheorem{definition}[theorem]{Definition}
\newtheorem{proposition}[theorem]{Proposition}
\newtheorem{oquestion}[theorem]{Open question}
\theoremstyle{remark}
\newtheorem{remark}[theorem]{Remark}

\renewcommand{\Im}{\operatorname{Im}}
\renewcommand{\Re}{\operatorname{Re}}
\DeclareMathOperator{\osc}{\mathrm{osc}}
\newcommand{\G}{\mathcal{G}}
\newcommand{\rI}{\mathrm{I}}
\newcommand{\rD}{\mathrm{D}}

\newcommand{\rS}{\mathrm{S}}

\DeclareMathOperator{\var}{Var}

\def\LipKd{{\mbox{\textsc{Lip(}$\kappa$\textsc{,}$\delta$\textsc{)}}}}
\def\ExpFat{{{\mbox{\textsc{Exp-Fat(}$\delta$\textsc{)}}}}}

\def\cD{\mathcal{D}}
\def\cF{\mathcal{F}}
\def\rH{\mathrm{H}}

\usepackage{bm}
\usepackage{color}

\title[Bipartite dimer model and Lorentz-minimal surfaces]{Bipartite dimer model: perfect t-embeddings and lorentz-minimal surfaces}

\author[Dmitry Chelkak]{Dmitry Chelkak$^\mathrm{a,b}$}

\author[Beno\^{\i}t Laslier]{Beno\^{\i}t Laslier$^\mathrm{c}$}

\author[Marianna Russkikh]{Marianna Russkikh$^\mathrm{d}$}

\thanks{\textsc{${}^\mathrm{A}$ D\'epartement de math\'ematiques et applications, \'Ecole Normale Sup\'erieure, CNRS, PSL University, 45 rue d'Ulm, 75005 Paris, France.}}

\thanks{\textsc{${}^\mathrm{B}$ On leave from St.~Petersburg Dept. of Steklov Mathematical Institute RAS, Fontanka 27, 191023 St.~Petersburg, Russia.}}

\thanks{\textsc{${}^\mathrm{C}$ Universit\'e de Paris, Sorbonne Universit\'e, CNRS, Laboratoire de Probabilit\'es, Statistiques et Mod\'elisations (LPSM), Paris, France}}

\thanks{\textsc{${}^\mathrm{D}$ Massachusetts Institute of Technology, Department of Mathematics, 77 Massachusetts Avenue, Cambridge, Massachusetts, 02139–4307}}

\thanks{\texttt{dmitry.chelkak@ens.fr}, \texttt{laslier@math.univ-paris-diderot.fr}, \texttt{russkikh@mit.edu}}

\begin{document}

\begin{abstract} This is the second paper in the series devoted to the study of the dimer model on t-embeddings of planar bipartite graphs. We introduce the notion of perfect t-embeddings and assume that the graphs of the associated origami maps converge to a Lorentz-minimal surface $\mathrm{S}_\xi$ as~$\delta\to 0$. In this setup we prove (under very mild technical assumptions) that the gradients of the {height correlation} functions converge to those of the Gaussian Free Field defined in the intrinsic metric of the surface~$\rS_\xi$. We~also formulate several open questions motivated by our work.
\end{abstract}



\maketitle

\tableofcontents

\newpage

\section{Introduction}

\subsection{General context and basic definitions}
 Let~$\G$ be a (big) weighted bipartite graph with the topology of the sphere; we call the two bipartite classes of its vertices black and white (and use the notation~$V(\G)=B\cup W$) and denote the weight of an edge~$e$ by~$\chi_e>0$. The dimer model is a random choice of a dimer cover~$\cD$ of~$\G$ (or, equivalently, of a perfect matching, i.e., a collection of edges that cover all vertices of~$\G$ exactly once), with probability proportional to~$\prod_{e\in\cD}\chi_e$. Below we always assume that~$\G$ admits a dimer cover even after removing two arbitrary vertices~$b\in B$ and~$w\in W$ from~it. {The planar dimer model has received a lot of attention, especially since Kasteleyn showed in 1960s that its partition function is given by the Pfaffian of a signed adjacency matrix of the underlying graph, which even simplifies in the bipartite case to the determinant of the so-called Kasteleyn matrix~$K_\R:\R^W\to\R^B$ with entries~$K(b,w)=\pm\chi_{(bw)}$ and appropriately chosen signs.} In particular, the planar dimer model is known to be intrinsically connected with a wide range of subjects varying from cluster algebras and representation theory to complex analysis and two-dimensional field theories, obviously not bypassing probability and various branches of combinatorics; we refer an interested reader to~\cite{kenyon-lectures,gorin-lectures} and references therein for more information.

 One of the central objects in the study of the planar bipartite dimer model is the so-called Thurston's \emph{height function,} which is defined as follows: choose a reference dimer cover~$\cD_0$ of~$\G$, superimpose it with~$\cD$, and note that all vertices of~$\G$ become covered by double-edges and cycles. This picture can be viewed as a topographic map and defines a function~$h:V(\G^*)\to\mathbb{Z}$ as follows: if~$(bw)$ is an edge of~$\G$ and~$(vv')=(bw)^*$ is the dual edge of~$\G^*$ oriented so that~$b$ is on the right, then \mbox{$h(v)-h(v')=\mathbbm{1}[(bw)\in\cD]-\mathbbm{1}[(bw)\in\cD_0]$}. In other words,~$h$ changes by~$\pm 1$ each time~$(vv')$ intersects one of the cycles constituting~$\cD\cup\cD_0$, {with the sign depending} on whether $(vv')^*\in\cD$ or $(vv')^*\in\cD_0$ and whether~$(vv')$ goes inside or outside this cycle. Thus, $h$~is a random function {defined} on vertices of the dual graph~$\G^*$, which depends on the choice of {the reference dimer cover}~$\cD_0$. However, it is easy to see that the fluctuations \mbox{$\hbar(v):=h(v)-\mathbb{E}[h(v)]$} do \emph{not} depend on the choice of~$\cD_0$. Given $v_1,\ldots,v_n\in V(\G^*)$ we set
 \begin{equation}\label{eq:Hn-def}
 H_n(v_1,\ldots,v_n)\ :=\ \mathbb{E}[\hbar(v_1)\ldots \hbar(v_n)]
 \end{equation}
 and call it the $n$-point {height} correlation function of the bipartite dimer model; note that~$H_1\equiv 0$ {because of the centering in the definition.}

 Since {the} seminal works of Kenyon~\cite{kenyon-gff-a,kenyon-gff-b} and Kenyon, Okounkov and Sheffield~\cite{kenyon-okounkov,kenyon-okounkov-sheffield} it is either rigorously known or predicted that in many setups the fluctuations~$\hbar$ of the height function become Gaussian in the so-called scaling limit~$\delta\to 0$ when a sequence of subgraphs~$\G^\delta$ of {a} periodic grid (e.g., square or honeycomb) of mesh size~$\delta$ approximate a given planar domain~$\Omega$; {e.g., see~\cite{BLR1,bufetov-gorin,chhita-johansson,gorin-lectures,kenyon-honeycomb,kenyon-lectures,laslier-21,zhongyang-li,petrov,russkikh-t,russkikh-h} and references therein.} It is worth emphasizing that the identification of the two-point correlation is {often} non-trivial: it should be thought of as the Green function of the Laplacian (with Dirichlet boundary conditions) in a certain non-trivial metric in~$\Omega$, which depends on the concrete setup; see~\cite{kenyon-okounkov,kenyon-okounkov-sheffield} or~\cite{kenyon-lectures,gorin-lectures} for details. Thus, 
 two questions arise: (a) why does the Gaussian structure appear in the limit and (b) how can one describe the relevant conformal structure of fluctuations in an `invariant' manner; i.e., not relying upon a concrete structure of the refining grids and, even more {ambitiously,} on their periodicity. This paper contributes to the study of these two questions by means of \emph{discrete complex analysis} techniques on \emph{t-embeddings} of planar bipartite graphs.

 The framework of t-embeddings or, equivalently, that of \emph{Coulomb gauges}, appeared recently in~\cite{KLRR} and in~\cite{CLR1}; the former paper mostly focuses on algebraic aspects of the dimer model while the latter is devoted to the study of the so-called t-holomorphic functions on t-embeddings, which can be thought of as a unifying approach to discrete holomorphic functions on regular grids, discrete harmonic functions on Tutte's barycentric embeddings and their gradients, and \mbox{s-holomorphic functions} on s-embeddings; see~\cite[Appendix]{CLR1}. In this paper we add two new ideas to the framework developed in~\cite{CLR1}:
 \begin{itemize}
 \item we introduce the notion of \emph{perfect} t-embeddings of finite planar bipartite graphs carrying the dimer model (or, equivalently, the notion of perfect Coulomb gauges in the terminology of~\cite{KLRR}; see Section~\ref{sub:finding-gauge});
 \item we interpret scaling limits of t-holomorphic functions using the conformal structure of (the limit of) the graphs of the associated origami maps viewed as surfaces in the Minkowski space~$\R^{2,2}$; see also~\cite[Section~2.7]{chelkak-s-emb} for a similar discussion in the planar Ising model context.
 \end{itemize}

 The first idea, a bit surprisingly, provides a setup in which one can prove the uniform boundedness of the so-called dimer coupling functions using only mild `non-degeneracy' assumptions~{\LipKd} and~{\ExpFat} that are discussed below. This uniform estimate of the coupling function allows us to apply the framework developed in~\cite{CLR1}.

 The second idea naturally (though rather unexpectedly) leads to the appearance of \emph{Lorentz-minimal surfaces} in the dimer model context; see also~\cite{chelkak-ramassamy} and~\cite{russkikh-cusp} for concrete examples. Though in this paper we take it as one of the `black box' assumptions (see Assumption~\ref{assump:Lorentz-min} below), we believe that this appearance should be a general phenomenon rather than a miraculous coincidence and support this viewpoint by an informal discussion given in Section~\ref{sub:questions}. From our perspective, developing this link between the bipartite dimer model and surfaces in the Minkowski space~$\R^{2,2}$ is an interesting research direction and we hope that Theorem~\ref{thm:main-GFF} given below is only the first step in understanding a bigger picture; see Section~\ref{sub:questions} for further comments.

\subsection{Perfect t-embeddings and assumptions in the main theorem}
\subsubsection{Definition of perfect t-embeddings} We now discuss a central definition of our paper, see Fig.~\ref{fig:p-emb} and Section~\ref{sub:origami} for details. Let~$(G,\chi)$ be a weighted bipartite graph with the topology of the sphere and a marked  face~$v_{\mathrm{out}}$. Recall that a collection of weights~$\widetilde{\chi}_e$ is called gauge equivalent to~$\chi_e$ if there exists two functions~$g^\tb:B\to\R_+$ and~$g^\tw:W\to\R_+$ (called gauge functions) such that~$\widetilde{\chi}_{(bw)}=g(b)\chi_{(bw)}g(w)$ for all edges~$(bw)$ of~$\G$. Replacing edge weights~$\chi_e$ by~$\widetilde{\chi}_e$ one does not change the law  of the dimer model on~$\G$ since the weight of each configuration is multiplied by the same factor. Denote by~$\G^*$ the augmentation of its dual graph at~$v_\mathrm{out}$, which means that~$\G^*$ contains a cycle of degree~$\deg v_\mathrm{out}$ replacing~$v_\mathrm{out}$ itself. We call~$\cT:\G^*\to\C$ a \emph{perfect t-embedding} of (the dual of) the weighted bipartite graph~$(\G,\chi)$ if

\begin{figure}
\includegraphics[height=0.7\textwidth]{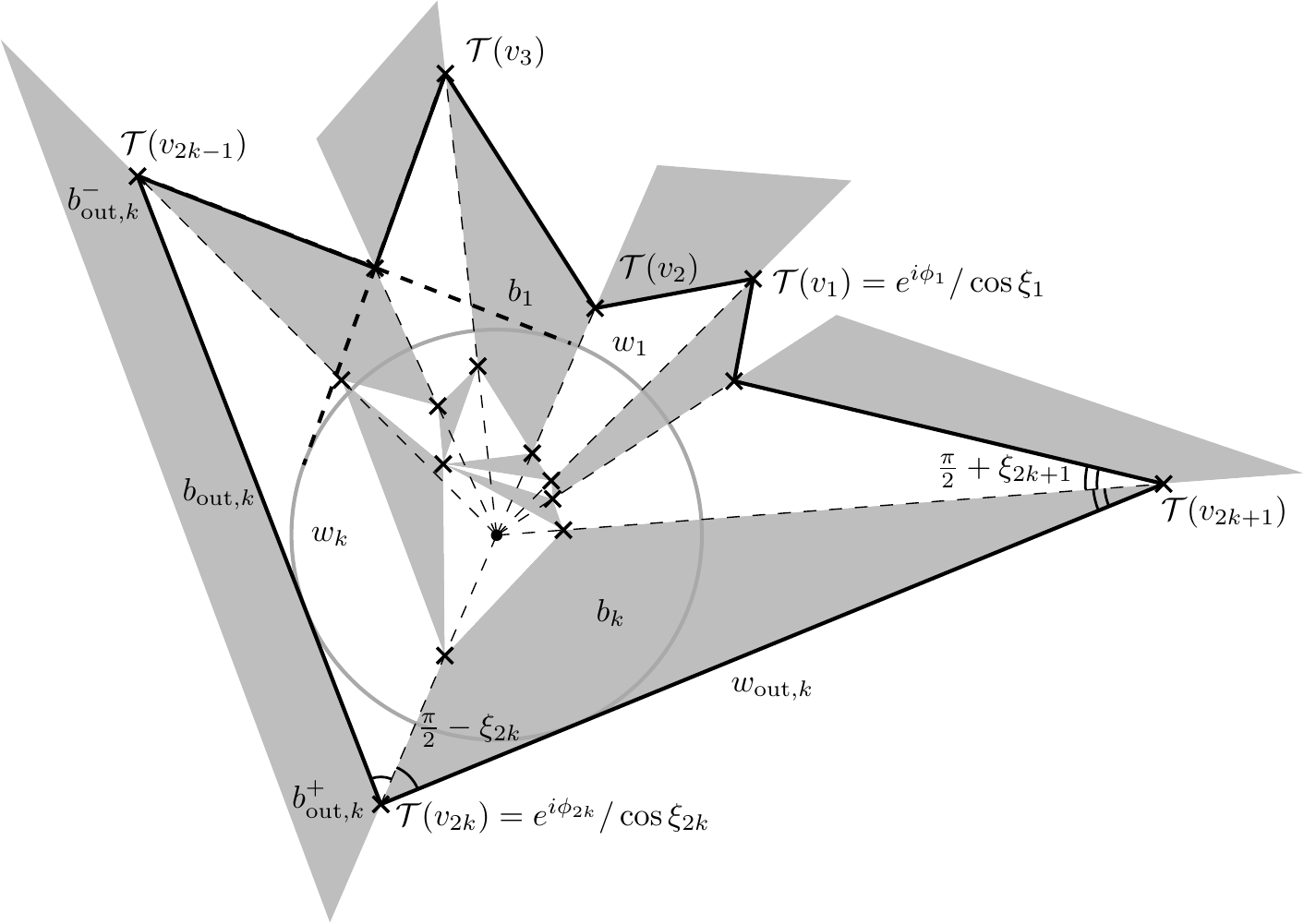}\hfill$\phantom{x}$

\vskip -0.15\textwidth

\includegraphics[clip, trim= 4.5cm 17.2cm 3.6cm 4.7cm, width=0.56\textwidth]{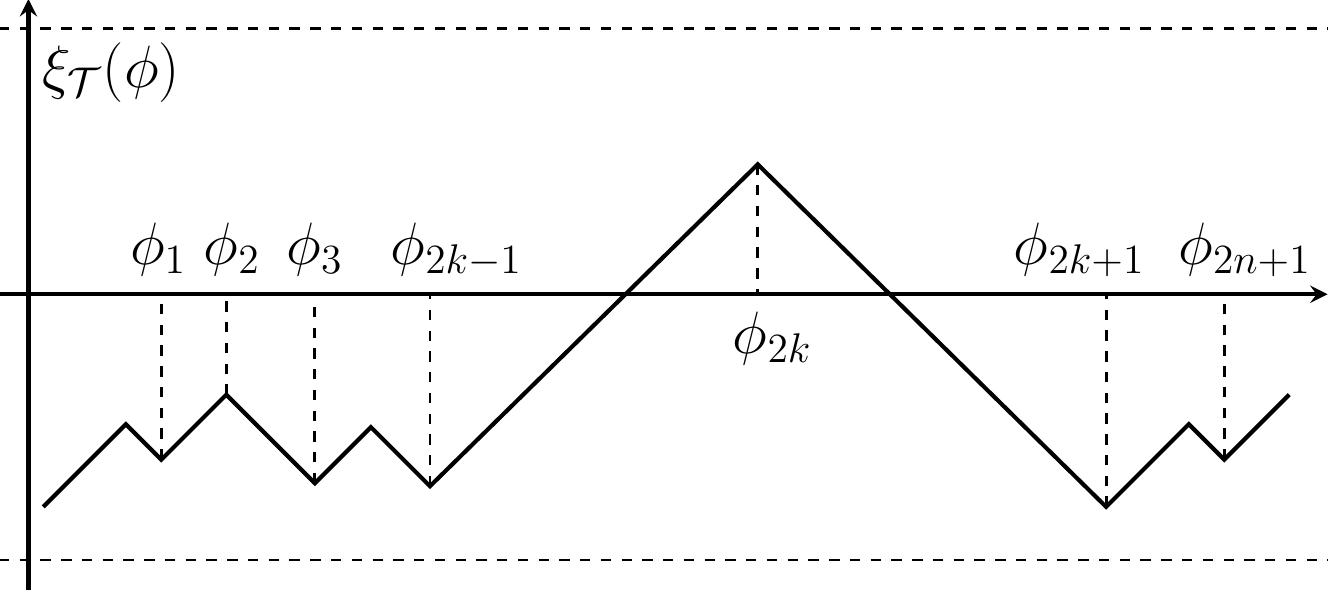}\hfill
\includegraphics[clip, trim= 4cm 21.4cm 11.6cm 2.4cm, width=0.36\textwidth]{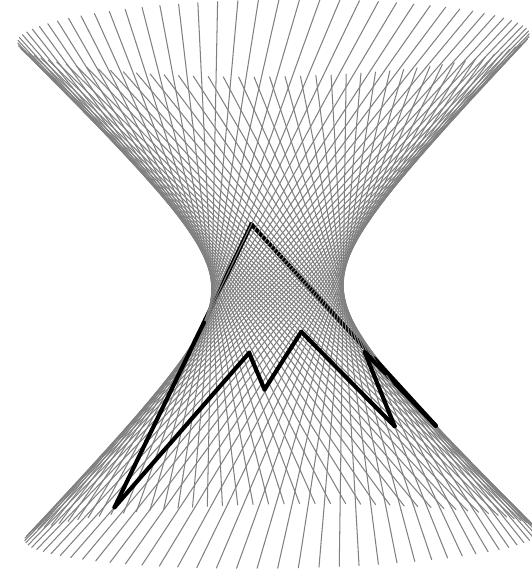}

\smallskip

\caption{\textsc{Top:} an example of a perfect t-embedding~${\cT}$ and the notation used throughout the paper; the dimer graph $\G$ is the octagonal prism and~$\deg v_\mathrm{out}=8$. \textsc{Bottom-left:} the graph of the corresponding piece-wise linear function~$\xi_\cT$. \textsc{Bottom-right:} the lift~{$\partial\rS_{\xi_\cT}\subset\R^{2,1}$} of the boundary of the domain~$\Omega_{\xi_\cT}{\subset \C\cong\R^2}$ onto the one-sheet hyperboloid~\eqref{eq:hyperboloid-def}.}
\label{fig:p-emb}
\end{figure}

\begin{itemize}
\item~$\cT$ is a \emph{proper} embedding, i.e., all edges of~$\cT(\G^*)$ are non-intersecting straight segments and all faces of~$\cT(\G^*)$ except the outer one are non-overlapping convex polygons;
\item the edge lengths~$|\cT(v')-\cT(v)|$ are \emph{gauge equivalent} to the original dimer weights~$\chi_{(bw)}$, where~$(vv')=(bw)^*$ denotes the edge of~$\G^*$ that is dual to the edges~$(bw)$ of the dimer graph~$\G=B\cup W$;
\item for each inner vertex~$v$ of~$\G^*$ the sum of angles of black faces adjacent to~$\cT(v)$ equals~$\pi$; following~\cite{KLRR,CLR1} we call this the \emph{angle condition};
\item the outer face of~$\cT$ is a (possibly, non-convex) \emph{tangential polygon} to the unit circle and all the non-boundary edges emanating from boundary vertices are \emph{bisectors} of the corresponding angles; see Fig.~\ref{fig:p-emb}.
\end{itemize}
The first three conditions already appeared in~\cite{KLRR} and~\cite{CLR1} {and define a generic t-embedding of finite or infinite graphs} while the {fourth} one is specific to this paper {and distinguishes \emph{perfect} t-embeddings in the finite case.}
In the terminology of~\cite{KLRR} we call the corresponding gauge functions a~\emph{perfect Coulomb gauge}. 
Let us emphasize that, unfortunately, we do not have a general result claiming the \emph{existence} of such gauges. However, we believe that they do exist for all sufficiently non-degenerate graphs~$\G$. For instance, it follows from~\cite{KLRR} (or from Theorem~\ref{thm:perfect-gauge} given below) that perfect Coulomb gauges always exist in the simplest case when~$\deg v_\mathrm{out}=4$; {see Section~\ref{sec:discussion} for a discussion.}

An important concept introduced in~\cite{KLRR} and~\cite{CLR1} is the notion of the \emph{origami map}~$\cO:\G^*\to\C$ associated with a (not necessarily perfect) t-embedding~$\cT$. We postpone its formal definition until Section~\ref{sub:origami} and only give an informal description here: to construct the mapping~$z\mapsto\cO(z)$ one folds the plane along each of the edges of~$\cT(\G^*)$; the complex coordinate of a given point $z\in\C$ after the folding is~$\cO(z)$. (This folding procedure is locally consistent due to the angle condition required at all inner vertices of~$\G^*$.) Note that the map~$z\mapsto\cO(z)$ cannot increase distances, i.e., is~$1$-Lipschitz.

It is easy to see that if~$\cT$ is a \emph{perfect} t-embedding, then~$\cO$ maps the outer face of~$\G^*$ onto a line; without loss of generality we assume that this is the real line in what follows. Moreover, the image of the outer faces under the mapping~$(\cT,\cO):\G^*\to\C\times\C\cong \R^{2,2}$ belongs to the one-sheet hyperboloid
\begin{equation}
\label{eq:hyperboloid-def}
\rH\ :=\ \{\,(z,\vartheta)\in\C\times\C\cong\R^{2,2}:\ |z|^2-|\vartheta|^2=1,\ \Im\vartheta=0\,\}\ \subset\ \R^{2,1},
\end{equation}
where we use the notation~$\R^{2,2}$ and~$\R^{2,1}:=\{(z,\vartheta):\Im\vartheta=0\}$ to emphasize that we equip these real vector spaces with the \emph{Minkowski} scalar product
\begin{equation}
\label{eq:Lorentz-product}
\langle(z_1,\vartheta_1),(z_2,\vartheta_2)\rangle\ :=\ \Re[z_1\overline{z_2}]-\Re[\vartheta_1\overline{\vartheta}_2]
\end{equation}
rather than with the usual Euclidean structure of~$\R^4$ or~$\R^3$, respectively.

\subsubsection{Domains~$\Omega_\xi$ and Lorentz-minimal surfaces~$\rS_\xi$} Given a~$1$-Lipschitz function~$\xi:\R/2\pi\mathbb{Z}\to (-\frac\pi 2,\frac\pi 2)$ we define a planar star-convex domain~$\Omega_\xi$ by
\begin{equation}
\label{eq:omega-xi-def}
\Omega_\xi :=\ \big\{z=\rho e^{i\phi},\ 0\le\rho< 1/\cos(\xi(\phi)),\ \phi\in\R/2\pi\mathbb{Z}\big\}\ \subset\ \C.
\end{equation}
It is easy to see that the image of a perfect t-embedding~$\cT$ in the complex plane is given by such a domain~$\Omega_\xi$, where~$\xi=\xi_\cT$ is a piece-wise linear function with slopes~$\pm 1$; see Fig.~\ref{fig:p-emb}. In our main result, Theorem~\ref{thm:main-GFF} given below, we consider a sequence of perfect t-embeddings~$\cT^\delta$ (of a sequence of growing bipartite graphs~$\G^\delta$ indexed by $\delta=\delta_m\to 0$) and assume that
\[
\xi^\delta\rightrightarrows\xi:\R/2\pi\mathbb{Z}\to (-\tfrac\pi 2,\tfrac\pi 2)\ \ \text{as}\ \ \delta\to 0,
\]
here and below~$`\rightrightarrows$' denotes the uniform convergence; note that the limit~$\xi$ is necessarily a $1$-Lipschitz function. Since the origami maps \mbox{$\cO^\delta:\Omega_{\xi^\delta}\to\C$} are $1$-Lipschitz, without loss of generality we can also assume that
\begin{equation}\label{eq:origami-conv}
\cO^\delta\rightrightarrows \vartheta:\overline{\Omega}_\xi\to\C\ \ \text{as}\ \ \delta\to 0,\ \ \text{uniformly on compact subsets of~$\Omega_\xi$},
\end{equation}
where the limit~$\vartheta$ is a~$1$-Lipschitz function on~$\Omega_\xi$ such that~$\vartheta(\partial\Omega_\xi)\subset\R$. We are now able to formulate the key assumption in Theorem~\ref{thm:main-GFF}.

\begin{figure}[t]
\centering \includegraphics[clip, trim=4.5cm 7.6cm 1cm 4cm, width=0.44\textwidth]{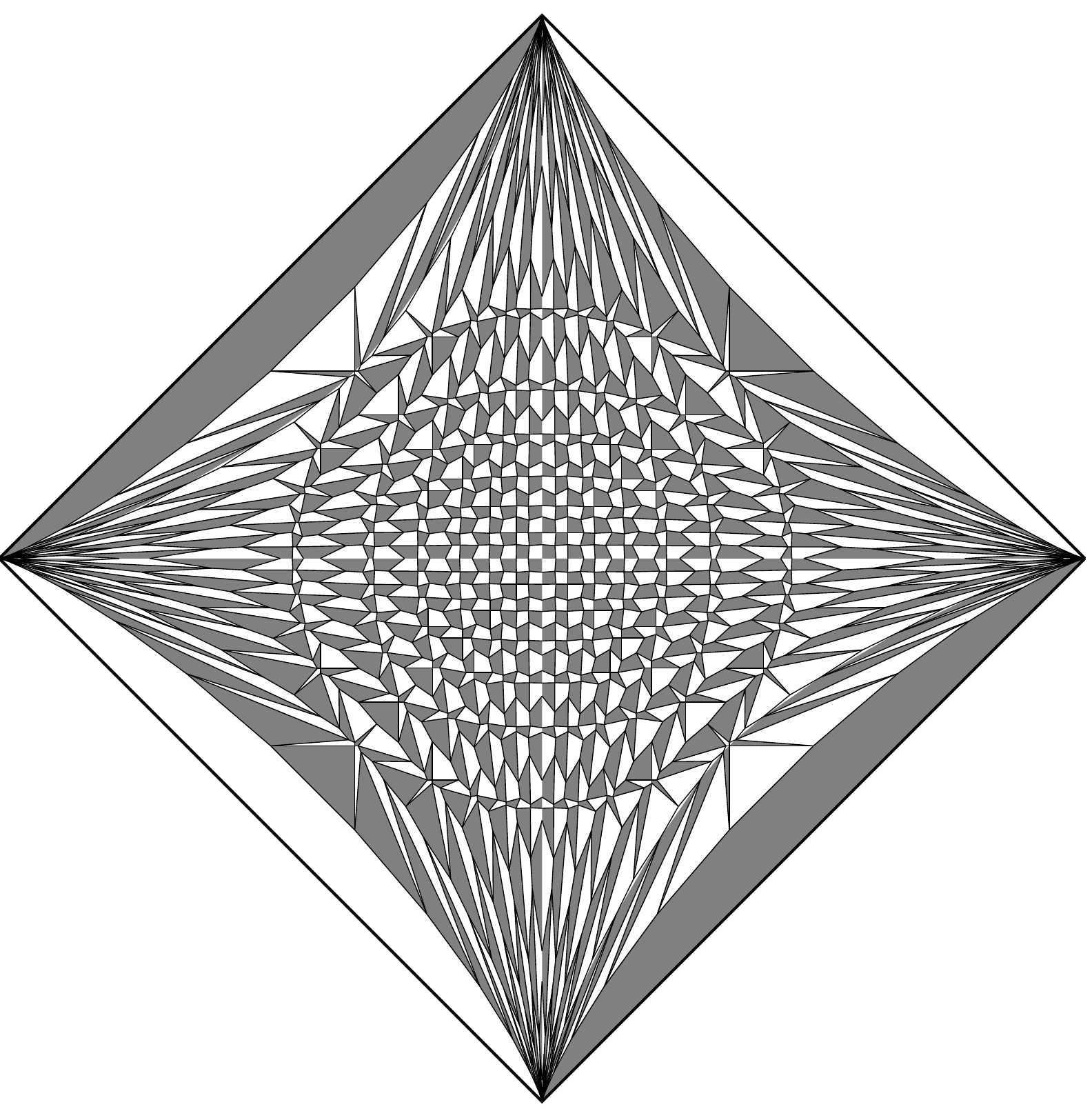}\hskip 0.06\textwidth
\includegraphics[clip, trim=4.8cm 22.8cm 11cm 2cm, width=0.48\textwidth]{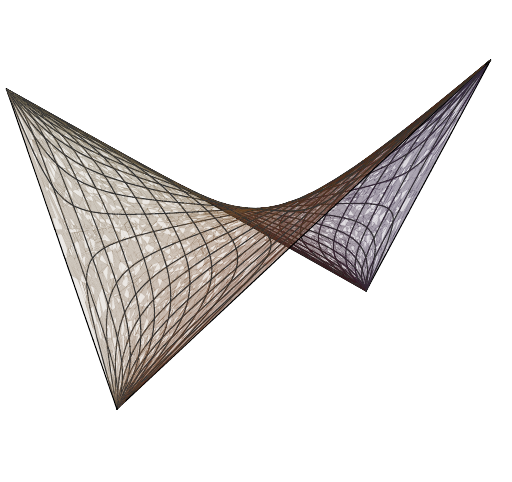}
\caption{\textsc{Left:} a perfect t-embedding of the homogeneous \emph{Aztec diamond}~$A_m$ of size $m=27$. The function~$\xi=\xi_{\cT_m}$ is given by~$\xi(\pi k)=-\frac\pi 4$ and~$\xi(\frac\pi 2+\pi k)=\frac\pi 4$ for all~$m$; see~\cite{chelkak-ramassamy}. \textsc{Right:} the Lorentz-minimal surface~${S_\xi\subset\Omega_\xi\times\R\subset \R^{2,1}}$ appearing in the limit~$m\to\infty$. Note that the paper~\cite{chelkak-ramassamy} does not contain neither a proof of the assumption~{\ExpFat} for these embeddings nor a discussion of {a} choice~$\delta=\delta_m$ that fits the assumption~$\LipKd$. {We consider~\cite{chelkak-ramassamy} more as a motivation of Assumption~\ref{assump:Lorentz-min} rather than as a particular case to which Theorem~\ref{thm:main-GFF} applies literally though we believe that} {one could check the technicalities using the results of~\cite{chhita-johansson-young}.}}\label{fig:Aztec}
\end{figure}

\begin{assumption}[{\bf Lorentz-minimality}] \label{assump:Lorentz-min}
The graph~$(z,\vartheta(z))_{z\in\Omega_\xi}$ of~$\vartheta$ is the Lorentz-minimal surface~$\rS_\xi$ solving the Plateau problem in the Minkowski space~$\R^{2,1}\subset \R^{2,2}$ for the boundary contour
\[
{\partial\rS_\xi\ :=\ \big\{(\cos\phi/\cos\xi(\phi),\sin\phi/\cos\xi(\phi),\tan\xi(\phi)),\ \phi\in\R/2\pi\mathbb{Z}\,\big\}\subset\R^{2,1}.}
\]
(In~particular, note that we require~$\Im\vartheta(z)=0$ for all~$z\in\Omega_\xi$.)
\end{assumption}
Below we give a list of properties of~$\vartheta$ that we will use in our paper. The reader can consider them as an equivalent reformulation of Assumption~\ref{assump:Lorentz-min}.
\begin{itemize}
\item The function $\vartheta$ is~$C^\infty$-smooth inside~$\Omega_\xi$ and~$\max_K|\nabla\vartheta|<1$ for all compact subsets~$K$ of~$\Omega_\xi$. In particular, $\rS_\xi$ is a space-like surface in the Minkowski space~$\R^{2,1}$ and the Riemannian metric on~$\rS_\xi$ inherited from~$\R^{2,1}$ is uniformly non-degenerate on compact subsets.
\item There exists a parametrization~$\zeta\in\mathbb{D}\to(z(\zeta),\vartheta(\zeta))\in\rS_\xi$ of the surface $\rS_\xi$ by the unit disc~$\mathbb{D}:=\{\zeta\in\C:|\zeta|<1\}$ which is simultaneously conformal and harmonic, i.e., for all~$\zeta\in\mathbb{D}$ one has
\begin{equation}
\label{eq:conf-param}
\pazeta z \cdot \pazeta\overline{z}=\pazeta\vartheta \cdot \pazeta\overline{\vartheta}\qquad\text{and}\qquad
\opazeta\pazeta z=\opazeta\pazeta\vartheta=0.
\end{equation}
(The first equation easily follows from the expression~\eqref{eq:Lorentz-product} for the scalar product in the Minkowski space~$\R^{2,2}$.) Note that in the actual setup of Theorem~\ref{thm:main-GFF} the function~$\vartheta$ is real-valued.
\item The mapping~$\zeta\mapsto z(\zeta)$ is a homeomorphism of~$\mathbb{D}$ and~$\Omega_\xi$ and one has $\zeta\to \partial\mathbb{D}$ if and only if~$z(\zeta)\to\partial\Omega_\xi$; note that this mapping is, in general (e.g., for polygonal domains~$\Omega_\xi$), \emph{not} continuous up to~$\partial\mathbb{D}$.
\end{itemize}

{Recall that while it looks rather unexpected, the Lorentz-minimality property of the limits~\eqref{eq:origami-conv} was recently observed in concrete examples~\cite{chelkak-ramassamy,russkikh-cusp} and should be a general phenomenon as we discuss in Section~\ref{sub:questions}.}

\subsubsection{Assumptions on possible degeneracies in perfect t-embeddings~$\cT^\delta$} We now move on to `technical' assumptions~{\LipKd} and~{\ExpFat} on (perfect) t-embeddings~$\cT^\delta$, which are adopted from the first paper~\cite{CLR1} in this series. Let us emphasize that we do \emph{not} rely on any kind of `uniformly bounded angles and/or lengths' properties of~$\cT^\delta$ and replace them by the assumption~{\ExpFat} formulated below. This assumption only requires that there are no macroscopic continua in~$\Omega_\xi$ filled by `exponentially flat' faces of~$\cT^\delta$; which looks plausible in many interesting setups, including those involving random maps. Moreover, in Theorem~\ref{thm:main-GFF} we only impose these assumptions on \emph{compact subsets} $K\subset \Omega_\xi$, allowing~$\cT^\delta$ to behave in an arbitrarily bad way near the boundary.

\begin{assumption}[\textbf{\LipKd}] \label{assump:LipKd}
Given $0<\kappa<1$ and $\delta > 0$ we say that a t-embedding $\cT=\cT^\delta$ satisfies assumption~{\LipKd} on a set $K\subset \Omega_{\cT}$ if
\[
|\cO(z')-\cO(z)|\ \le\ \kappa\cdot|z'-z|\quad \text{for all $z,z'\in K$ such that $|z-z'|\ge\delta$},
\]
where~$\cO$ stands for the origami map associated with~$\cT$.
\end{assumption}
It is worth noting that this assumption rather \emph{defines} a `nice' scale~$\delta=\delta_m$ for a given t-embedding~$\cT=\cT_m$ than takes it as an input, this is why we omit the superscript~$\delta$ from the notation. In particular, it directly follows from the definition of the origami map~$\cO$ that the maximal diameter of faces of~$\cT$ lying inside~$K$ cannot exceed~$\delta$; let us emphasize that~$\delta$ can be much bigger than this maximal diameter. Recall that in this paper we consider a sequence of t-embeddings such that the limit~\eqref{eq:origami-conv} of the origami maps satisfies the condition~$\max_K|\nabla\vartheta|<1$ on all compacts~$K\subset\Omega_\xi$. This implies that the assumption~{\LipKd} always holds (on compact subsets of~$\Omega_\xi$) for a certain \emph{choice} of~$\delta$ associated to these t-embeddings. However, the forthcoming assumption~{\ExpFat} requires to fix this dependence of~$\delta$ and~$\cT=\cT^\delta$; this is why one cannot simply drop~{\LipKd} in Theorem~\ref{thm:main-GFF}.

Given~$\rho>0$, we say that a polygon is `$\rho$-fat' if it contains a disc of radius~$\rho$. If a t-embedding~$\cT$ has non-triangular faces, then we say that~$\cT^\tb_\spl$ (resp.~$\cT^\tw_\spl$) is a splitting of~$\cT$ if it is obtained by splitting all black (resp., white) faces of~$\cT$ into triangles. The following assumption originated in~\cite[Section~6.5]{CLR1}.

\begin{assumption}[\textbf{\ExpFat}] \label{assump:ExpFat-general}
We say that a sequence of \mbox{t-embeddings} $\cT^\delta=\cT_{\xi^\delta}$ satisfy the assumption~{\ExpFat} as~$\delta\to 0$ on a common region~$K\subset\Omega_\xi$ if the following properties hold for each~$\beta>0$:\\[2pt]
(i) there exists splitting~$\cT^{\tb,\delta}_\spl$ such that, if one removes all `$\exp(-\beta\delta^{-1})$-fat' white faces and all `$\exp(-\beta\delta^{-1})$-fat' black triangles from $\cT_\spl^{\tb,\delta}$, then the size of {any} remaining (in~$K$) vertex-connected components tends to zero as~$\delta\to 0$;\\[2pt]
(ii) there exists splittings~$\cT^{\tw,\delta}_\spl$ satisfying the same condition.
\end{assumption}
Let us emphasize that, contrary to the assumption~{\LipKd} which can be used to define the scales~$\delta=\delta_m\to 0$ for a given sequence of perfect \mbox{t-embeddings} $\cT_m$ and a given compact~$K\subset\Omega_\xi$, the assumption~{\ExpFat} requires that this dependence is fixed in advance.

\subsection{Main result and organization of the paper} We are now ready to formulate the main result of this paper. Informally speaking, we prove that in the setup described above the gradients of $n$-point correlation functions~\eqref{eq:Hn-def} converge to those of the standard Gaussian Free Field in the intrinsic metric of the Lorentz-minimal surface~$\rS_\xi$ that was introduced in Assumption~\ref{assump:Lorentz-min}.

\begin{theorem}\label{thm:main-GFF} Let~$\cT^\delta$, $\delta=\delta_m\to 0$ be a sequence of perfect t-embeddings of weighted bipartite graphs~$\G_m$, $m\to\infty$. Assume that these embeddings satisfy Assumptions~\ref{assump:Lorentz-min}--\ref{assump:ExpFat-general} discussed above. Let~$(v_{1,1},\ldots,v_{n,1})$ and~$(v_{1,2},\ldots,v_{n,2})$ be two $n$-tuples of pairwise distinct inner points of the domain~$\Omega_\xi$ and assume that~$v_{k,r}=z(\zeta_{k,r})$, $\zeta_{k,r}\in\mathbb{D}$, under the conformal parametrization of the Lorentz-minimal surface~$\rS_\xi\subset\R^{2,1}\subset\R^{2,2}$. Let~$v_{k,r}^\delta$ be an approximation of $v_{k,r}$ on the t-embedding~$\cT^\delta$ and {recall the definition~\eqref{eq:Hn-def} of the height correlation functions~$H_n=H_n^\delta$ of the bipartite dimer model on~$\cT^\delta$.} Then,
\begin{align*}
&\textstyle \sum_{r_1,\ldots,r_n\in\{1,2\}}(-1)^{r_1+\ldots r_n}\,H_n^\delta(v^\delta_{1,r_1},\ldots,v^\delta_{n,r_n})\\[4pt]
&\qquad\qquad  \to\ \textstyle \sum_{r_1,\ldots,r_n\in\{1,2\}}(-1)^{r_1+\ldots r_n}\,G_{\mathbb{D},n}(\zeta_{1,r_1},\ldots,\zeta_{n,r_n})
\end{align*}
as~$\delta\to 0$, where~$G_{\mathbb{D},n}$ denotes the $n$-point correlation function of the Gaussian Free Field in the unit disc~$\mathbb{D}$ with Dirichlet boundary conditions. The convergence is uniform provided that the points~$v_{k,1}$ stay at a definite distance from each other and from $\partial\Omega_\xi$ and that the same is true for the points~$v_{k,2}$.
\end{theorem}
\begin{remark} It is worth noting that along the way we do prove that the correlation functions~$H_n^\delta$ are uniformly Lipschitz on compact subsets of~$\Omega_\xi$. Thus, the result of Theorem~\ref{thm:main-GFF} can be rigorously formulated as the convergence of their gradients~$\nabla_{v_1}\ldots\nabla_{v_n}H_n^\delta(v_1,\ldots,v_n)$. However, Theorem~\ref{thm:main-GFF} does \emph{not} imply a stronger statement on the convergence of functions~$H_n^\delta$ themselves to~$G_{\mathbb{D},n}$. For instance, a simple example that is \emph{not} excluded by our proof is the convergence of the height fluctuations~$\hbar^\delta$ to the sum of the GFF on~$\rS_\xi$ and an independent global random variable. The reason is that, unlike in previously known proofs, we do \emph{not} identify the limit, as~$\delta\to 0$, of dimer coupling functions~$K^{-1}_{\cT^\delta}(w^\delta,b^\delta)$ and rely upon a more robust framework developed in~\cite{CLR1} instead. Then, we manage to deduce a uniform estimate~$K^{-1}_{\cT^\delta}=O(1)$ from general principles but only in the situation when \emph{one} of the points~$w^\delta$, $b^\delta$ is allowed to approach the boundary of a perfect t-embedding~$\cT^\delta$ and not \emph{both} of them. Though we believe that this issue can be resolved by a more careful analysis, let us nevertheless say explicitly that we are not aware of such an argument at the moment.
\end{remark}
\begin{remark} Though the existence of perfect t-embeddings of big dimer graphs remains an open question, it is tempting to discuss a general phenomenology that Theorem~\ref{thm:main-GFF} suggests provided that a picture similar to Fig.~\ref{fig:Aztec} appears for other polygonal domains drawn on periodic grids. We expect that \emph{frozen regions} must be always collapsed to vertices of the polygonal domain~$\Omega_\xi$ since this is the only way how the (always non-degenerate) conformal structure on~$\rS_\xi$ can produce zero correlation functions. In a similar spirit, we believe that \emph{gaseous bubbles} must be collapsed to points inside~$\Omega_\xi$ and that the corresponding Lorentz-minimal surface~$\rS_\xi$ has a \emph{cusp} there. In particular, we do not think that Theorem~\ref{thm:main-GFF} can be applied literally in presence of gaseous bubbles though a good portion of its proof should survive: e.g., we do not essentially rely upon the fact that the limit~\eqref{eq:origami-conv} is real-valued in our analysis. We refer the reader to Section~\ref{sub:questions} for further comments.
\end{remark}

The rest of the  paper is organized as follows. We collect preliminary facts on t-embeddings and the so-called t-holomorphic functions in Section~\ref{sec:prelim}. The main reference for this material (except Section~\ref{sub:hol-on-S}) is our recent paper~\cite{CLR1}. We keep the presentation self-contained in what concerns the formulation of required statements; note that the most important input that we take from~\cite{CLR1} is the a priori regularity theory for t-holomorphic functions discussed in Section~\ref{sub:regularity}. Section~\ref{sec:main-proof} is the core part of the paper, in which we prove Theorem~\ref{thm:main-GFF}. More precisely, in Section~\ref{sub:boundedness} we prove the key uniform estimate for dimer coupling functions on perfect t-embeddings and then follow a scheme suggested in~\cite{CLR1} in order to complete the proof of Theorem~\ref{thm:main-GFF} in Section~\ref{sub:convergence}. As in Section~\ref{sec:prelim}, this proof is written in a self-contained manner and does not require {any prior} knowledge of~\cite{CLR1}. Section~\ref{sec:discussion} (and, especially, Section~\ref{sub:questions}) is mostly devoted to a discussion of a more general perspective. Coming back to the existence of perfect t-embeddings, in Section~\ref{sub:finding-gauge} we prove that this essentially amounts to finding a Coulomb gauge that sends all boundary vertices of~$\G^*$ onto the hyperboloid~\eqref{eq:hyperboloid-def}: the result is automatically a \emph{proper} embedding of~$\G^*$ (see Theorem~\ref{thm:perfect-gauge} for an exact statement). This reduces the existence question to solving a certain system of biquadratic equations in~$2\deg v_\mathrm{out}$ real variables and provides an additional support to the relevance of the setup developed in our paper.

\addtocontents{toc}{\protect\setcounter{tocdepth}{1}}
\subsection*{Acknowledgements} We would like to thank Mikhail Basok for many helpful comments on a preliminary version of this article.
D.\,C. is grateful to Olivier Biquard for an advice on the Lorentz geometry and to Richard Kenyon and Sanjay Ramassamy for inspiring discussions. M.\,R. would also like to thank Alexei Borodin for useful discussions.

D.\,C. is the holder of the ENS–MHI chair funded by~MHI. The research of D.\,C. and B.\,L. was partially supported by the ANR-18-CE40-0033 project DIMERS. The research of M.\,R. is supported by the Swiss NSF grants P400P2-194429 and P2GEP2-184555 and also partially supported by the NSF Grant DMS-1664619.
\addtocontents{toc}{\protect\setcounter{tocdepth}{2}}

\section{Preliminaries}\label{sec:prelim}

\subsection{Perfect t-embeddings and the origami map} \label{sub:origami}
We begin by briefly recalling the setup of \mbox{\emph{t-embeddings}} (see~\cite{CLR1}) or, equivalently, \emph{Coulomb gauges} (see~\cite{KLRR}) of finite bipartite graphs carrying the dimer model. Let $\G$ be a finite bipartite graph with the topology of the sphere and a marked vertex $v_{\operatorname{out}}$ of its dual graph~$\G^*$; we call the bipartite classes black and white, and use the notation~$B\cup W=V(\G)$ for them. The graph obtained from $\G^*$ by removing~$v_{\operatorname{out}}$ and replacing it by a cycle of length $\operatorname{deg} v_{\operatorname{out}}$ is called its {augmentation} at $v_{\operatorname{out}}$; by construction, all the new vertices replacing~$v_\mathrm{out}$ have degree~$3$. In the following, with a slight abuse of terminology we will use the symbol $\G^*$ and the name dual graph for the augmented dual at $v_\mathrm{out}$.

\begin{definition}[{see~\cite{CLR1} and~\cite{KLRR}}] \label{def:t-emb}
A finite t-embedding~$\cT$ of a planar graph \mbox{$\G^*=(B\cup W)^*$} with the topology of the sphere and a marked vertex $v_{\operatorname{out}}$ is a proper embedding of its augmentation at $v_{\operatorname{out}}$ into the complex plane such that all edges are straight segments, all faces (except the outer one) are convex polygons, and that for each inner vertex~$v$ of~$\G^*$ the sum of angles of black faces adjacent to $\cT(v)$ equals~$\pi$; see Fig.~\ref{fig:p-emb}.
\end{definition}

For an edge~$(bw)$ of~$\G$, let~$(bw)^*$ or simply $bw^*$ be the corresponding edge of~$\G^*$ oriented so that~$b$ is to the right. Given an oriented edge~${e^* = (vv')}$ of~$\G^*$ denote ${d\cT( e^* ) := \cT( v') - \cT(v)}$ and let
\[
K(b,w)=K_\cT(b,w)\ :=\ d\cT(bw^*)\,.
\]
It is easy to see that the angle condition given in Definition~\ref{def:t-emb} implies that~$K$ is a Kasteleyn matrix for the bipartite dimer model on~$\G$ with weights \mbox{$\chi_{bw}:=|d\cT(bw^*)|$} (e.g., see~\cite{KLRR} for a background). This means that the partition function of the model is equal to~$|\det K|$ and that
\begin{align}
\mathbb{P}\left[\begin{array}{l}\text{edges~$(b_1w_1),\ldots,(b_kw_k)$ simultaneously}\\
\text{present in a random dimer configuration} \end{array}\right]\qquad &\notag \\  \
\ =\ \det [K^{-1}(w_j,b_k)]_{j,k=1}^n\cdot\textstyle \prod_{k=1}^n K(b_k,w_k)\,.&
\label{eq:P-k-dimers}
\end{align}
The next definition did not appear in~\cite{KLRR,CLR1} and is specific to this paper.
\begin{definition} \label{def:p-emb}
We call a finite t-embedding~$\cT$ from Definition~\ref{def:t-emb} a \emph{perfect} \mbox{t-embedding} or a \emph{p-embedding} if its outer face is a tangential polygon (possibly, non-convex) to the unit circle and all non-boundary edges emanating from boundary vertices are bisectors of the corresponding angles; see Fig.~\ref{fig:p-emb}.
\end{definition}
Let $\cT$ be a perfect t-embedding. We use the notation~$\partial B\subset B$ and $\partial W\subset W$ for the set of \emph{boundary faces} of~$\cT$, i.e., those faces that are adjacent to the outer face. Now consider the bisectors of the explementary angles of the boundary polygon of $\cT$ and glue to each boundary edge an outer face with a color different from the color of the incident boundary face, such that non-boundary edges adjacent to boundary vertices lie on these bisectors; see Fig.~\ref{fig:p-emb}. Denote the sets of thus obtained black and white faces by $\partial_{\operatorname{out}}B$ and $\partial_{\operatorname{out}}W$, respectively. Denote~$\partial_\mathrm{out}\cT:=\partial_\mathrm{out}B\cup\partial_\mathrm{out} W$ and let
\[
\overline{B}:= B\cup \partial_\mathrm{out} B\quad\text{and}\quad
\overline{W}:= W\cup\partial_\mathrm{out} W\,.
\]
We use the following notation {throughout the paper,} see Fig.~\ref{fig:p-emb}:
\begin{itemize}
	\item The number of outer faces of $\cT$ {(i.e., the degree of~$v_\mathrm{out}$)} is $2n$;
\item the boundary vertices of~$\cT$ are labeled~$v_1,\ldots,v_{2n}$ in the counterclockwise order (as usual, we set~$v_{2k+1}:=v_1$ etc);
\smallskip
\item the edge~$~(v_{2k-1}v_{2k})$ separates the faces~$w_k\in\partial W$ and~$b_{\mathrm{out},k}\in\partial_\mathrm{out}B$\,;
 the edge~$(v_{2k}v_{2k+1})$ separates~$b_k\in\partial B$ and~$w_{\mathrm{out},k}\in \partial_\mathrm{out}W$\,;
\smallskip
\item $\phi_k:=\arg \cT(v_k)$ denotes the direction from the origin towards~$\cT(v_k)$ with the convention that~$\phi_{k+1}-\phi_k\in (0,\pi)$ for all~$k=1,\ldots,2n$ and therefore~$\phi_{2n+1}=\phi_1+2\pi$;

\smallskip
\item the inner half-angle of the boundary polygon of~$\cT$ equals
$\frac\pi2-\xi_{2k}$ at a vertex~$\cT(v_{2k})$ and~$\frac\pi2+\xi_{2k+1}$ at a vertex~$\cT(v_{2k+1})$,
where~$\xi_k\in(-\frac\pi 2,\frac\pi 2)$.
\end{itemize}
It is easy to see that
\begin{equation}\label{eq:phi-phi=xi-xi}
\phi_{2k-1}-\xi_{2k-1}=\phi_{2k}-\xi_{2k} \quad \text{and}\quad \phi_{2k}+\xi_{2k}=\phi_{2k+1}+\xi_{2k+1}
\end{equation}
for all~$k=1,\ldots,n$ since these expressions are nothing but the directions to the points at which the lines containing the boundary edges of~$\cT$ touch the unit circle. In particular, these identities imply that~$\sum_{k=1}^n(\phi_{2k}-\phi_{2k-1})=\pi$.

\begin{remark}\label{rem:xi-k}
{The domain covered by a perfect \mbox{t-embedding}~$\cT$ {admits} an explicit description in terms of the angles $\xi_k$ introduced above. Namely, if we define a function $\xi_\cT:\mathbb{\R}/2\pi\mathbb{Z}\to (-\frac\pi 2,\frac\pi 2)$ by setting $\xi_\cT(\phi_k)=\xi_k$ and extending it linearly along the segments~$[\phi_k,\phi_{k+1}]$, then~$\cT$ covers the domain~$\Omega_{\xi_\cT}$ defined by~\eqref{eq:omega-xi-def}; see also Fig.~\ref{fig:p-emb}. In particular, $\cT(v_k) = e^{i\phi_k}/\cos \xi_k$ and the convergence assumption made in Theorem~\ref{thm:main-GFF} implies that the angles $\xi_k=\xi_k^\delta$ remain isolated from~$\pm \frac\pi 2$ as~$\delta\to 0$.}
\end{remark}
Following~\cite{CLR1}, we call a function~$\eta:\overline{B}\cup \overline{W}\to\mathbb{T}:=\{\eta\in\C:|\eta|=1\}$ an \emph{origami square root function} if
\begin{equation}
\label{eq:dT=etaeta}
d\cT(bw^*)\ \in\ \overline{\eta}_b\overline{\eta}_w\R\quad \text{for all adjacent~$b\in \overline{B}$ and~$w\in \overline{W}$;}
\end{equation}
the existence of such functions easily follows from the angle condition in Definition~\ref{def:t-emb} (note that the same condition holds around all boundary vertices of~$\cT$). Let us emphasize that the values of~$\eta$ are defined \emph{up to the signs} only. In most places (e.g., in the crucial Lemma~\ref{lem:t-hol-true}) these signs are irrelevant. However, sometimes (e.g., in~\eqref{eq:def-Fb=K-1}) we need to make a choice of one of the two possible values. Below we assume that these choices are made once forever for all faces of~$\cT=\cT^\delta$, in an arbitrary manner. Note also that the stronger condition $d\cT(bw*) \in \overline{\eta}_b\overline{\eta}_w \R_+$ does not admit any solution~$\eta$ if the graph~$\G$ contains faces of degree in~$4\mathbb{Z}$; cf.~\cite[Remark~2.5]{CLR1}.

In what follows we denote by~$z$ the complex coordinate in~$\C$; note that one can view~$dz$ as an extension of the discrete differential form~$d\cT$ defined on edges of~$\G^*$ into the domain~$\Omega_{\xi_\cT}\subset\C$ covered by a (perfect) t-embedding~$\cT$.

\begin{definition}\label{def:O} Let~$\eta$ be an origami square root function on a t-embedding~$\cT$. The associated origami map $\cO$ is a primitive of the differential form
\begin{equation}
\label{eq:def-O}
d \cO (z) := \begin{cases}
  \overline{\eta}_b^2 \,d \bar{z} & \text{if $z$ {belongs to a} black face $\cT(b)$,}\\
  \eta_w^2 \,dz & \text{if $z$ {belongs to a} white face $\cT(w)$.}
\end{cases}
\end{equation}
Note that $d\cO$ also can be seen as a discrete $1$-form on edges of $\G^*$ by setting
\begin{equation}
\label{eq:dO-on-T}
d \cO ( bw^*)\ :=\  
\overline{\eta_b^2d \cT( bw^*)}\ =\ \eta_w^2 \,d\cT( bw^*).
\end{equation}
\end{definition}

\begin{remark}
	As discussed in~\cite{CLR1,KLRR}, one can think about the origami map in the following way: to construct {the mapping $z\mapsto\cO(z)$} one folds the {complex} plane along each of the edges of a t-embedding~$\cT$; this procedure is locally consistent due to the angle condition in Definition~\ref{def:t-emb}. The complex coordinate of a given point~$z$ after the folding is performed is~$\cO(z)$.
\end{remark}

It is easy to see that the condition~\eqref{eq:dT=etaeta} defines the function~$\eta^2$ uniquely up to transforms~$\eta^2_b\mapsto \lambda\eta^2_b$, $\eta^2_w\mapsto\overline{\lambda}\eta^2_w$, where~$\lambda\in\mathbb{T}$. This means that the origami map~$z\mapsto \cO(z)$ is defined uniquely up to rotations and translations.

If~$\cT$ is a perfect t-embedding, then one can fix this ambiguity in the definition of the origami square root function and of the origami map by choosing the convention
\begin{equation}
\label{eq:eta=perfect}
\begin{array}{rcl}
\overline{\eta}_{b_k}^2&\!\!=\!\!&-ie^{i(\phi_{2k}+\xi_{2k})}\ =\ -ie^{i(\phi_{2k+1}+\xi_{2k+1})},\\
\overline{\eta}_{w_k}^2&\!\!=\!\!&ie^{i(\phi_{2k}-\xi_{2k})}\ =\ ie^{i(\phi_{2k-1}-\xi_{2k-1})},
\end{array}
\end{equation}
for all boundary faces~$b_k\in\partial B$ and~$w_k\in\partial W$, respectively, and, moreover,
\[
\cO(v_k)\ =\ \tan\xi_k\quad \text{for all}\ \ k=1,\ldots,2n.
\]
In particular,~$\cO$ maps the boundary of the domain~$\Omega_{\xi_\cT}$ in~$\R$. (The fact that all boundary segments of a perfect t-embedding are mapped onto the same line trivially follows from the bisector condition in Definition~\ref{def:O}.) A graph of the mapping~$z\mapsto \cO(z)$ plotted over the domain~$\Omega_{\xi_\cT}$ is a piece-wise linear surface in the space~$\R^2\times\R^2$, which is required to converge to the \emph{Lorentz-minimal} surface~$\rS_\xi\subset\Omega_\xi\times \R^2$ in the setup of Theorem~\ref{thm:main-GFF}. This necessarily implies that~$\rS_\xi\subset\Omega_\xi\times \R$ since this condition holds at the boundary of~$\partial\Omega_\xi$.

\subsection{Inverse Kasteleyn matrix and t-holomorphic functions}\label{sec:def_t_holom}
We now recall the definition and basic properties of the so-called t-holomorphic functions on t-embeddings. This notion was recently introduced and studied in~\cite{CLR1}; we refer the reader to this paper for more details. In what follows we only consider t-holomorphic functions with standard boundary conditions in the terminology of~\cite{CLR1}; in fact, the only functions to which we apply the general theory in this paper come from the entries of the inverse Kasteleyn matrix (see equation~\eqref{eq:def-Fb=K-1} below) and thus satisfy these boundary conditions.

Given a collection~$\mathfrak{p}$ of faces of~$\cT$ (in fact, for the purposes of this paper it is enough to assume that~$\mathfrak{p}$ is a single face; see definition~\eqref{eq:def-Fb=K-1} below) we denote the union of remaining faces~$\cT_\mathfrak{p}$ and call it a \emph{punctured} t-embedding.

\begin{definition}[t-holomorphic functions, `fake' complex values] \label{def:t-hol-fake}
Let~$\cT$ be a finite t-embedding and~$\mathfrak{p}\subset B$. We say that a function~$F^\tw_\frb:\overline{W}\to\C$ is t-black-holomorphic on~$\cT_\mathfrak{p}$ and satisfies standard boundary conditions if
\begin{itemize}
\item $F^\tw_\frb(w)\in\eta_w\R$ for all~$w\in \overline{W}$ and, moreover,~$F^\tw_\frb(w)=0$ if~$w\in W_\mathrm{out}$;

\item for all inner black faces~$b \in B\smallsetminus\mathfrak{p}$ the following identity holds:
\begin{equation}
\label{eq:t-hol-fake-def}
\textstyle \int_{\partial b}F^\tw_\frb d\cT\ :=\ -\sum_{w\in W:\,w\sim b} F^\tw_\frb(w)d\cT(bw^*)\ =\ 0.
\end{equation}
\end{itemize}
Given~$\mathfrak{p}\subset W$, one defines t-white-holomorphic on~$\cT_\mathfrak{p}$ functions~$F^\tb_\frw:\overline{B}\to C$ satisfying standard boundary conditions similarly.
\end{definition}
Since~$d\cT(bw^*)$ is nothing but the entry of the Kasteleyn matrix~$K$, the latter condition is equivalent to say that~$KF^\tw_\frb=0$ on~$B\smallsetminus \mathfrak{p}$. If~$\mathfrak{p}=\{b\}$, this condition is clearly satisfied by the function~$K^{-1}(\,\cdot\,,b)$, where we formally define~$K^{-1}(w,b):=0$ if $w\in W_\mathrm{out}$. Therefore, for each~$b\in B$ the function
\begin{equation}
\label{eq:def-Fb=K-1}
F^\tw_b(w):=\overline{\eta}_bK^{-1}(w,b)
\end{equation}
is t-black-holomorphic on~$\cT_{\{b\}}$ and satisfies standard boundary conditions. A similar statement holds for the functions~$F^\tb_w(b):=\overline{\eta}_wK^{-1}(w,b)$ with~$w\in W$.

In~\cite{CLR1} we argued that instead of working with `essentially real' (i.e., having a prescribed complex phase) values~$K^{-1}(w,b)$ of~$F^\tw_\frb(w)$, it is convenient to introduce their `complexifications'; see Lemma~\ref{lem:t-hol-true} below. To do this on general t-embeddings, we introduce their \emph{splittings}: a black splitting~$\cT^\tb_\spl$ is obtained from~$\cT$ by adding diagonal segments in all its black faces of degree at least~$4$ so as they are decomposed into triangles.
As in~\cite[Section~5]{CLR1}, we still view $\cT_\spl^\tb$ as a t-embedding by interpreting each added segment as a white $2$-gon with zero angles, note that this does not break the angle condition; in particular one can extend the origami square root function~$w\mapsto\eta_w$ onto these 2-gons. Let $\G_\spl^\tb=B_\spl^\tb\cup W_\spl^\tb$ denote the associated dimer graph. White~splittings~$\cT^\tw_\spl$ are defined similarly; see~\cite[Figure~6]{CLR1}.

\begin{lemma}[t-holomorphic functions, `true' complex values]\label{lem:t-hol-true} Let~$\mathfrak{p}\subset B$. A function~$F_\frb^\tw$ is t-black-holomorphic on~$\cT_\mathfrak{p}$ and satisfies standard boundary conditions if and only if there exists a function~$F_\frb^\tb:B^\tb_\spl\smallsetminus \mathfrak{p}\to\C$ such that
\begin{equation}
\label{eq:t-hol-true-def}
F_\frb^\tw(w)\ =\ \Pr(F_\frb^\tb(b),\eta_w\R)\quad \text{if}~w\sim b,\ \ w\in \overline{W}{}_{\!\spl}^\tb,\ \ b\in B_\spl^\tb\smallsetminus\mathfrak{p}.
\end{equation}
(Note that we also extend~$F_\frb^\tw$ from~$W$ to~$W^\tb_\spl$ along the way). A similar statement holds for t-white-holomorphic functions and white splittings~$\cT^\tw_\spl$.
\end{lemma}
\begin{proof} See~\cite[Lemma~3.4]{CLR1} for the equivalence of the condition~\eqref{eq:t-hol-fake-def} and the existence of a value~$F_\frb^\tb(b)\in\C$ such that~\eqref{eq:t-hol-true-def} holds in the case when~$b$ is a triangular face of~$\cT$ and~\cite[Proposition~5.4]{CLR1} for a general case.
\end{proof}
It is worth noting that the `true' complex values~$F_\frb^\tb$ (resp., $F_\frw^\tw$) of t-holomorphic functions depend on the choice of the splitting~$\cT^\tb_\spl$ (resp., $\cT^\tw_\spl$) of a given t-embedding~$\cT$. However, this dependence is local: if one changes a splitting of a single face~$b$, only the values of~$F_\frb^\tb$ on the faces obtained from~$b$ change. Moreover, a priori regularity estimates for t-holomorphic functions (see Theorem~\ref{thm:F-precomp} below) eventually imply that these values are actually almost independent of the way in which faces are split, at least for bounded t-holomorphic functions and at faces lying in the bulk of~$\cT$. Hereafter when we talk about values of $F_\frw^\tw$ (resp., $F_\frb^\tb$) we assume that a white (resp., black) splitting was chosen arbitrarily if not explicitly specified.

\begin{remark} \label{rem:F=O(1)-equiv}
It is clear that a uniform boundedness of the `true' complex values~$F_\frb^\tb$ of a t-black-holomorphic function implies the boundedness of their projections~$F_\frb^\tw$. However, the converse implication is not straightforward as faces of t-embeddings~$\cT=\cT^\delta$ can have very small angles. Still, this holds provided that~$\cT^\delta$ satisfy  \emph{both} assumptions~{\LipKd} and {\ExpFat} as \mbox{$\delta\to 0$}; see~\cite[Lemma~6.20]{CLR1}. In particular, in the setup of Theorem~\ref{thm:main-GFF} a uniform estimate~$F_\frb^{\tw,\delta}=O(1)$ on all compacts~$K\subset \Omega_\xi$ implies the same for~$F_\frb^{\tb,\delta}$. A~similar equivalence holds for t-white-holomorphic functions.
\end{remark}

In Section~\ref{sub:boundedness} (see also Lemma~\ref{lem:max-t-hol} below) we will also need to define the `true' complex values at~$\partial_\mathrm{out} \cT$ for t-holomorphic functions satisfying standard boundary conditions. Each outer boundary face~$b_{\mathrm{out},k}$ is incident to two boundary faces~$w_{\mathrm{out},k},w_{\mathrm{out},{k-1}}$ at which~$F^\tw_\frb$ vanishes, thus one clearly cannot use the same definition~\eqref{eq:t-hol-true-def} at~$b\in\partial_{\mathrm{out}}B$. In order to overcome this difficulty we define two values~$F^\tb_\frb(b^\pm_{\mathrm{out},k})\in\C$ instead of a single one so that
\begin{equation}
\label{eq:F-pa-out-def}
{\Pr(F^\tb_\frb(b^\pm_{\mathrm{out},k}),\eta_{w_k}\R)=F^\tw_\frb(w_k),\quad\begin{array}{l}
\Pr(F^\tb_\frb(b^+_{\mathrm{out},k},\eta_{w_{\mathrm{out},k}}\R)=0,\\[4pt]
\Pr(F^\tb_\frb(b^-_{\mathrm{out},k},\eta_{w_{\mathrm{out},k-1}}\R)=0,
\end{array}}
\end{equation}
We proceed similarly in order to define the `true' boundary values~$F^\tw_\frw(w_{\mathrm{out},k}^\pm)$ of t-white-holomorphic functions satisfying standard boundary conditions.

\begin{proposition} \label{prop:closed-forms-single}
Let~$F_\frb$ (resp., $F_\frw$) be a t-black- (resp., t-white-) holomorphic function defined on a punctured t-embedding~$\cT_\mathfrak{p}$, $\mathfrak{p}\subset B$ (resp., $\mathfrak{p}\subset W$) and satisfying standard boundary conditions. Then, the following identities hold on edges of~$\cT^\tb_\spl$ (resp., $\cT^\tw_\spl$) not adjacent to~$\mathfrak{p}$:
\begin{equation}
\label{eq:closed-forms-edges}
{2}F^\tw_{\frb} \,d \cT\,=\,F_\frb^\tb\, d \cT + \overline{F}{}_\frb^\tb\, d \cO,\qquad
  {2}F_\frw^\tb \,d\cT\,=\,F_\frw^\tw \,d \cT + \overline{F}{}_\frw^\tw \,d \overline{\cO},
\end{equation}
where we use the notation $F^\tw((bw)^*):=F^\tw(w)$ and~$F^\tb((bw)^*):=F^\tb(b)$.

Moreover, the discrete differential forms~\eqref{eq:closed-forms-edges} can be viewed as closed piece-wise constant differential forms
\begin{equation}
\label{eq:closed-forms-plane}
F_\frb^{\tb}{(z)dz} + \overline{F_\frb^{\tb}{(z)}} d \cO{(z)},\qquad F_\frw^{\tw}{(z)d z} + \overline{F_\frw^{\tw}{(z)}}d \overline{\cO{(z)}}
\end{equation}
defined in the domain on complex plane covered by~$\cT$, where in the first expression one sets
\[
\begin{array}{ll}
F_\frb^\tb(z):=F_\frb^\tb(b) & \text{if $z\in\cT(b)$ for~$b\in B^\tb_\spl\smallsetminus \mathfrak{p}$,}\\[4pt]
F_\frb^\tb(z):=F_\frb^\tb(\bc)&  \text{if $z\in\cT(w)$ for~$w\in W$, where  $w\sim\bc\in B^\tb_\spl\smallsetminus\mathfrak{p}$;}
\end{array}
\]
the differential form~\eqref{eq:closed-forms-plane} does not depend on the choice of a face~$\bc$ adjacent to~$w$ in the second line (and similarly for the second expression in~\eqref{eq:closed-forms-plane}).
\end{proposition}
\begin{proof}
See~\cite[Lemma~3.8]{CLR1}.
\end{proof}
It is well known (e.g., see a discussion in~\cite[Section~7.3]{CLR1}) that entries of the inverse Kasteleyn matrix -- or t-holomorphic functions themselves -- often has a rather unstable behavior with respect to the microscopic structure of the boundary of (big) dimer graphs. In the next proposition we introduce more stable objects, namely the primitives of products of two such functions.
\begin{proposition} \label{prop:closed-forms-double}
Let~$F_\frb$ (resp., $F_\frw$) be a t-black- (resp., t-white-) holomorphic function defined on a punctured t-embedding~$\cT_\mathfrak{p}$ and satisfying standard boundary conditions. The following identity holds on edges not adjacent to~$\mathfrak{p}$:
\begin{equation}\label{eq:FFdT=}
F_\frb^\tw F_\frw^\tb\,d\cT \,=\, \tfrac{1}{2}\Re \big( F_\frb^\tb F_\frw^\tw \,d \cT  +  \overline{F}{}^\tb_\frb F^\tw_\frw\, d \cO \big),
\end{equation}
where we use the same convention as in~\eqref{eq:closed-forms-edges}. Similarly to~\eqref{eq:closed-forms-plane}, this form can be viewed as a closed piece-wise constant differential form
\begin{equation}
\label{eq:FFdT-plane}
{\tfrac{1}{2}}\Re \big( F_\frb^\tb{(z) F_\frw^\tw{(z)}dz}+ \overline{F}{}_\frb^\tb{(z)}F^\tw_\frw{(z)}\, d \cO{(z)} \big),
\end{equation}
defined in the complex plane. Moreover, it vanishes on boundary edges of~$\cT$.
\end{proposition}
\begin{proof} See \cite[Proposition~3.10]{CLR1}.
\end{proof}
Recall that the `true' complex values of t-holomorphic functions~$F_b(\cdot)$, $b\in B$, and~$F_w(\cdot)$, $w\in W$, defined in Lemma~\ref{lem:t-hol-true} can be thought of as a `complexification' of the dimer coupling function~$K^{-1}(w,b)$ with respect to \emph{one} of its arguments. The next proposition describes the result of this procedure applied to \emph{both} arguments of $K^{-1}$ simultaneously.

\begin{proposition} \label{prop:Fpmpm-def}
There exist four complex-valued functions~$\FF{\pm\pm}=\FF{\pm\pm}_{\cT^\delta}$
defined on pairs $(\bc,\wc)$ of faces $\bc\in B^\tb_\mathrm{spl}$ and $\wc\in W^\tw_\mathrm{spl}$ obtained (by splitting) from non-adjacent faces of a t-embedding~$\cT=\cT^\delta$ such that

\smallskip

\noindent (i) $\FF{--}(\bc,\wc)=\overline{\FF{++}(\bc,\wc)}$ and $\FF{+-}(\bc,\wc)=\overline{\FF{-+}(\bc,\wc)}$\,;

\smallskip

\noindent (ii) if $w\sim\bc\ne b$ and $b\sim \wc\ne w$, then
\begin{align*}
K^{-1}(w,b)\ =\ &\tfrac{1}{4}\big(\FF{++}(\bc,\wc)+\eta_b^2\FF{+-}(\bc,\wc)\\
&\phantom{\tfrac{1}{4}\big(\FF{++}(\bc,\wc)}
+\eta_w^2\FF{-+}(\bc,\wc)+\eta_w^2\eta_b^2\FF{--}(\bc,\wc)\big)\,,\\
F_b^\tb(\,\cdot\,)\ =\ &\tfrac{1}{2}\big(\overline{\eta}_b\FF{++}(\,\cdot\,,\wc)+\eta_b\FF{+-}(\,\cdot\,,\wc)\big)\,,\\
F_w^\tw(\,\cdot\,)\ =\ &\tfrac{1}{2}\big(\overline{\eta}_w\FF{++}(\bc,\,\cdot\,)+\eta_w\FF{-+}(\bc,\,\cdot\,)\big)\,.
\end{align*}

\smallskip

\noindent (iii) Moreover, for all~$\eta\in\C$ the function
\[
\tfrac{1}{2}\big(\overline{\eta}\FF{++}(\,\cdot\,,\wc)+\eta\FF{+-}(\,\cdot\,,\wc)\big)\ \ \text{(\ resp.,}\ \
\tfrac{1}{2}\big(\overline{\eta}\FF{++}(\bc,\,\cdot\,)+\eta\FF{-+}(\bc,\,\cdot\,)\big)\ \text{)}
\]
is t-black- (resp., t-white-) holomorphic away from~$\wc$ (resp., from~$\bc$).
\end{proposition}
\begin{proof}
See~\cite[Proposition 3.12]{CLR1}.
\end{proof}

\subsection{Primitives of t-holomorphic functions and T-graphs}
\begin{definition}
\label{def:I[F]}
Let~$F$ be a t-holomorphic (either t-black- or t-white-) function on a punctured t-embedding~$\cT_\mathfrak{p}$, satisfying standard boundary conditions. We denote by~$\rI_\C[F]$ the primitive of the differential form~\eqref{eq:closed-forms-plane}. Given a direction~$\alpha\in\mathbb{T}$ we also denote by
\begin{equation}
\label{eq:I-alpha-def}
\rI_{\alpha\R}[F]\ :=\ \Pr(\rI_\C[F],\alpha\R)\ =\ \alpha\Re(\overline{\alpha}\,\rI_\C[F])
\end{equation}
the projection of~$\rI_\C[F]$ onto the line~$\alpha\R$.
\end{definition}
Recall that the form~\eqref{eq:closed-forms-edges} can be also written as~$2F_\frb^\tw d\cT$ (resp., $2F_\frw^\tb d\cT$) if we consider it on edges of the t-embedding~$\cT$ only; from this perspective the primitive~$\rI_\C[F]$ is defined on vertices of~$\cT$, including boundary ones.

Note however that this primitive is well defined only locally: namely, if~$F$ is t-holomorphic on a punctured t-embedding~$\cT_\mathfrak{p}$, then~$\rI_\C[F]$ can have additive monodromies around faces from~$\mathfrak{p}$. In particular, if~$b\in B$, then the monodromy of the t-black-holomorphic function~$F_b$ around~$b$ equals
\[
\textstyle \oint_{\partial b} 2F_b^\tw d\cT\ =\ -2\sum_{w\in W: w\sim b}\overline{\eta}_b K^{-1}(w,b) K(b,w)\ =\ -2\overline{\eta}_b.
\]
Similarly, the additive monodromy of~$\rI_\C[F_w]$ around~$w\in W$ equals~$2\overline{\eta}_w$. In particular, the functions $\rI_{i\overline{\eta}_b\R}[F_b]$ and $\rI_{i\overline{\eta}_w\R}[F_w]$ are well-defined.

In fact, it is useful to view the projections~\eqref{eq:I-alpha-def} not as functions defined on vertices of the t-embedding~$\cT$ itself but as functions defined on vertices of an associated \emph{T-graph}~$\cT+\alpha^2\cO$ for t-white-holomorphic functions and $\cT+\alpha^2\overline{\cO}$ for t-black-holomorphic ones. We refer the reader to~\cite[Section~4 and Section~5]{CLR1} for a formal discussion of these T-graphs and only briefly mention their basic properties below.
It is easy to see from~\eqref{eq:def-O} that the images of faces of~$\cT$ have the following form:


\smallskip

\noindent -- in the T-graph~$\cT+\alpha^2\cO$\,:
\begin{itemize}
\item the image of $b\in B$ is a translate of the segment $2\Pr( \cT(b), {\alpha}\overline{\eta}_b \R)$\,;
\item the image of $w\in W$ is a translate of the polygon $(1\!+\!\alpha^2\eta_w^2)\cT(w)$\,;
\end{itemize}

\noindent -- in the T-graph~$\cT+\alpha^2\overline{\cO}$\,:
\begin{itemize}
\item the image of $b\in B$ is a translate of the polygon $(1\!+\!\alpha^2\eta_b^2)\cT(b)$\,;
\item the image of $w\in W$ is a translate of the segment $2\Pr( \cT(w), {\alpha}\overline{\eta}_w\R)$\,.
\end{itemize}

For generic~$\alpha\in\mathbb{T}$ all faces of, say, the T-graph $\cT+\alpha^2\cO$ are non-degenerate and one can consider a natural directed (continuous time) \emph{random walk}~$X_t$ on its vertices: once arriving at a vertex lying inside an edge $(\cT+\alpha^2\cO)(b)$, the walk~$X_t$ jumps to one of the endpoints of this edge so that both $\Re X_t$ and~$\Im X_t$ are martingales and~$\var \mathrm{Tr} X_t =t$. This definition can be extended to all~$\alpha\in\mathbb{T}$ (i.e., to the case when some of the faces of~$\cT+\alpha^2\cO$ are degenerate) by continuity in~$\alpha$; see~\cite[Section~4.1 and Definition~5.5]{CLR1}.

The harmonicity of functions defined on vertices of a T-graph is understood in the usual sense, i.e., with respect to the aforementioned random walk. It is worth noting that one can naturally extend each harmonic function from vertices to all points lying on edges of the T-graph by linearity.
Given a harmonic function on the T-graph~$\cT+\alpha^2\cO$ (resp., on~$\cT+\alpha^2\overline{\cO}$) one can define its \emph{discrete gradient}~$\rD[H]$ by specifying that
\[
dH=\rD[H](b)dz\qquad \text{(\ resp.,}\ dH=\rD[H](w)dz\ \text{)}
\]
along each segment $(\cT\!+\!\alpha^2\cO)(b)$, $b\in B$, (resp. $(\cT\!+\!\alpha^2\overline{\cO})(w)$, $w\in W$) of this T-graph; note that these segments never degenerate.

The next proposition shows that \emph{for each~$\alpha\in \mathbb{T}$} t-holomorphic functions on~$\cT$ can be equivalently defined (at least locally) as discrete gradients of \mbox{$\alpha\R$-valued} harmonic functions on the corresponding T-graph.

\begin{proposition} \label{prop:t-hol=grad}
Let~$\alpha\in\mathbb{T}$ and $\cT$ be a t-embedding. The following holds:

\smallskip

\noindent (i) If $F_\frw$ be a t-white-holomorphic function defined on (a piece of) $\cT$, then its primitive~$\rI_{\alpha}[F_\frw]$ is harmonic on the \mbox{T-graph~$\cT+\alpha^2\cO$} with respect to the directed random walk described above.
Similarly, the primitives~$\rI_{\alpha\R}[F_\frb]$ of t-black-holomorphic functions are harmonic on the T-graphs~$\cT+\alpha^2\overline{\cO}$.

\smallskip

\noindent (ii) Vice versa, if~$H$ is an $\alpha\R$-valued harmonic function on the T-graph \mbox{$\cT+\alpha^2\cO$} (resp., $\cT+\alpha^2\overline{\cO}$), then its discrete gradient~$\rD[H]$ defines the values~$F_\frw^\tb$ (resp., $F_\frb^\tw$) of a t-white- (resp., t-black-) holomorphic function on~$\cT$.
\end{proposition}
\begin{proof}
See~\cite[Section~4.2]{CLR1}.
\end{proof}

{It is clear that discrete harmonic functions~$H$ on T-graphs satisfy the \emph{maximum principle}: the maximum of~$H$ in a given subdomain of a T-graph is attained at the boundary. Since the transforms~$\cT(z)\mapsto (\cT+\alpha^2\cO)(z)$ and \mbox{$\cT(z)\mapsto (\cT+\overline{\alpha}^2\cO)(z)$} are known not to create overlaps (e.g., see~\cite[Proposition~4.3]{CLR1}), the maximum principle for the primitives~$\rI_{\alpha\R}[F]$ of t-holomorphic functions can be equivalently formulated on the t-embedding~$\cT$ itself instead of passing to T-graphs~$\cT+\alpha^2\cO$ or~$\cT+\alpha^2\overline{\cO}$.

A similar discussion applies to regularity properties of the functions~$\rI_{\alpha\R}[F]$: if~$\cT$ satisfies the assumption~{\LipKd}, then
\begin{equation}
\label{eq:distortion}
1-\kappa\ \le\ \frac{|(\cT\!+\!\alpha^2\cO)(v')-(\cT\!+\!\alpha^2\cO)(v)|}{|\cT(v')-\cT(v)|}\ \le\ 1+\kappa
\end{equation}
for all~$v',v$ with $|\cT(v')-\cT(v)|\ge \delta$. Thus, regularity properties of~$\rI_{\alpha\R}[F]$ on scales above~$\delta$ can be equivalently formulated on t-embeddings or T-graphs.}

\begin{proposition}
\label{prop:H-Holder} There exist constants $\beta=\beta(\kappa)>0$, $C=C(\kappa)>0$ and~$\rho_0=\rho_0(\kappa)>0$ such that for all t-embeddings~$\cT$ satisfying the assumption~{\LipKd} and for all~$\alpha\in\mathbb{T}$ the following H\"older-type estimate holds:

if~$\cT$ covers a ball~$B(v,R)$ of radius~$R>r\ge \rho_0\delta$ and~$H$ is a discrete real-valued harmonic function on the associated T-graph~$\cT+\alpha^2\cO$, then
\[
\osc_{B(v,r)}H\ \le\ C(r/R)^\beta\osc_{B(v,R)}H,
\]
where~$\osc_V H:=\max_V H-\min_V H$.
\end{proposition}
\begin{proof} See~\cite[Proposition~6.13]{CLR1} and~\eqref{eq:distortion}.
\end{proof}
The proof of~\cite[Proposition~6.13]{CLR1} is based on a uniform -- above scale~$\delta$ from the assumption~{\LipKd} -- ellipticity estimate for the random walks on T-graphs discussed above; see~\cite[Proposition~6.4]{CLR1}. This also implies the so-called uniform crossing estimates for these random walks above scale~$\delta$ and hence the discrete Harnack principle.
\begin{proposition}\label{prop:Harnack} In the same setup, for each~$\rho<1$ there exists a constant~$c(\rho)=c(\rho,\kappa)>0$ such that the following holds: if a positive function~$H$ is defined on vertices of~$\cT$ in the interior of a disc~$B(v_0,r)$ and is harmonic on the associated T-graph~$\cT+\alpha^2\cO$ (or, similarly, on~$\cT+\alpha^2\overline{\cO}$), then
\[
\min\nolimits_{B(v_0,\rho r)}H\ \ge\ c(\rho)\cdot \max\nolimits_{B(v_0,\rho r)}H
\]
provided that~$(1-\rho)r\ge\mathrm{cst}\cdot\delta$ for a constant~$\mathrm{cst}$ depending on~$\kappa$ only.
\end{proposition}
\begin{proof}
See~\cite[Proposition~6.9]{CLR1} and~\eqref{eq:distortion}.
\end{proof}
Since t-holomorphic functions are discrete derivatives of harmonic functions on T-graphs one can wonder whether it is possible to improve the H\"older-type regularity estimate provided by Proposition~\ref{prop:H-Holder} to a similar Lipschitz-type estimate, maybe under additional assumption on t-embeddings under consideration. We discuss this question in the next section.

\subsection{A priori regularity theory for t-holomorphic functions}\label{sub:regularity}
In this section we recall several properties of t-holomorphic functions on t-embeddings developed in~\cite{CLR1}. Summing it up we call them the a priori regularity theory for t-holomorphic functions. It consists of two key parts:
\begin{itemize}
\item Bounded t-holomorphic functions on t-embeddings~$\cT^\delta$ satisfying the assumption~{\LipKd} are (H\"older) equicontinuous and thus form pre-compact families; see Theorem~\ref{thm:F-precomp} below.
\item If both~{\LipKd} and~{\ExpFat} hold, then t-holomorphic functions admit an a priori Harnack--type estimate via their harmonic primitives~$\rI_{\alpha\R}$ discussed in the previous section. In particular, a uniform bound on the oscillations of these primitives implies a uniform bound on t-holomorphic functions themselves; see Theorem~\ref{thm:F-via-H} below.
\end{itemize}
We also need a preliminary lemma.
\begin{lemma}\label{lem:max-t-hol} Let~$F_\frb$ be a t-black-holomorphic function defined on (a part of) a t-embedding $\cT$. For each~$\alpha\in\mathbb{T}$ there exists a directed nearest-neighbor random walk on black faces of~$\cT$ such that~$\Pr(F^\bullet_\frb(\cdot),\alpha\R)$ is a martingale with respect to this random walk. In particular, the function~$\Re[\overline{\alpha}F^\tb_\frb]$ satisfies the discrete maximum principle for each~$\alpha\in \mathbb{T}$ and hence the same is true for~$|F^\tb_\frb|$\,. Similar statements hold for t-white-holomorphic functions.

Moreover, this maximum principle remains true if one includes into the consideration the values of t-holomorphic functions at~$\partial_{\mathrm{out}}\cT$ defined by~\eqref{eq:F-pa-out-def}.
\end{lemma}
\begin{proof}
See~\cite[Section~4.3]{CLR1}. Rather surprisingly, these random walks can be thought of as backward random walks on the T-graphs discussed in the previous section; see~\cite[Proposition~4.17]{CLR1} for the exact statement.

The proof of the maximum principle easily extends to~$\partial_\mathrm{out}\cT$ since~\eqref{eq:F-pa-out-def} gives a consistent definition of `true' complex values of a t-holomorphic function around each boundary vertex of~$\cT$.
\end{proof}

\begin{theorem} \label{thm:F-precomp}
Let~$F^\delta_\frb$, $\delta={\delta_m}\to 0$, be a sequence of t-black- (or, similarly, t-white-) holomorphic functions defined on t-embeddings~$\cT^\delta$ that contain a fixed open set~$U\subset\C$ and satisfy the assumption~{\LipKd} with~$\kappa=\kappa(K)<1$ on compact subsets~\mbox{$K\subset U$}. Assume that the functions~$F^{\tb,\delta}_\frb$ are uniformly bounded on compact subsets of~$U$. Then, the family~$\{F^{\tb,\delta}_\frb\}$ is pre-compact in the topology of the uniform convergence on compact subsets of~$U$.

If we additionally assume that the origami maps~$\cO^\delta$ converge (uniformly on compacts) to a function~$\vartheta:U\to \mathbb{C}$ as~$\delta \to 0$, then, for each subsequential limit~$\fb$ (resp.,~$\fw$) of the functions~$F^{\tb,\delta}_\frb$ (resp,~$F^{\tw,\delta}_\frw$), the differential form
\begin{equation}
\label{eq:closed-forms}
\fb(z)dz+\overline{\fb(z)}d\vartheta(z)\quad \text{(\ resp.,\ \ } \fw(z)dz+\overline{\fw(z)}d\overline{\vartheta(z)}\ \text{)}
\end{equation}
is closed.
\end{theorem}
\begin{proof} See~\cite[Corollary~6.14(i)]{CLR1} and~\cite[Proposition~6.15]{CLR1}. The main ingredient of the proof is that, under the assumption~{\LipKd}, the backward random walks on T-graphs mentioned in the proof of Lemma~\ref{lem:max-t-hol} satisfy uniform crossing estimates on all scales above~\mbox{$\mathrm{cst}(\kappa)\delta$}; see~\cite[Section~6.3]{CLR1}. Therefore, the functions~$F^\delta_\frb$ are H\"older equicontinuous on scales above~$\mathrm{cst}(\kappa)\delta$ and thus the pre-compactness of the family~$\{F^\delta_\frb\}$ follows from a version of the Arzel\'a--Ascoli theorem.

The fact that differential forms~\eqref{eq:closed-forms} are closed follows from a similar fact for their discrete versions discussed in Proposition~\ref{prop:closed-forms-single}.
\end{proof}

\begin{corollary} \label{cor:Fpmpm-hol-in-zeta}
In the setup of {Theorem~\ref{thm:main-GFF},}
let~$U_1,U_2$ be disjoint open subsets of~$\Omega_\xi$ and assume that the complexified dimer coupling functions~$\FF{\pm\pm}_{\cT^\delta}$ (see Proposition~\ref{prop:Fpmpm-def}) are uniformly bounded on compact subsets of~$U_1\times U_2$ as~$\delta={\delta_m}\to 0$. Then, the families~$\{\FF{\pm\pm}_{\cT^\delta}\}$ are precompact in the topology of uniform convergence on compacts and the following holds for each subsequential limits~$\ff{\sspm\sspm}:U_1\times U_2\to\C$ of these functions:\\[4pt]
(i)\phantom{ii} $\ff{\ssm\ssm}(z_1,z_2)=\overline{\ff{\ssp\ssp}(z_1,z_2)}$ and~$\ff{\ssp\ssm}(z_1,z_2)=\overline{\ff{\ssm\ssp}(z_1,z_2)}$\,;\\[4pt]
(ii)\phantom{i} $\ff{\ssp\sspm}(z_1,z_2)dz_1+\ff{\ssm\ssmp}(z_1,z_2)d\vartheta(z_1)$ are closed forms for each~$z_2\in U_2$;\\[4pt]
(iii) $\ff{\sspm\ssp}(z_1,z_2)dz_2+\ff{\ssmp\ssm}(z_1,z_2)d\vartheta(z_2)$ are closed forms for each~$z_1\in U_1$.
\end{corollary}
\begin{proof} It follows from Proposition~\ref{prop:Fpmpm-def}(iii) that for each~$\wc$ both functions
\[
\tfrac{1}{2}\big(\,\FF{++}_{\cT^\delta}(\,\cdot\,,\wc)+\FF{+-}_{\cT^\delta}(\,\cdot\,,\wc)\,\big)\quad \text{and}\quad \tfrac{i}{2}\big(\,\FF{++}_{\cT^\delta}(\,\cdot\,,\wc)-\FF{+-}_{\cT^\delta}(\,\cdot\,,\wc)\,\big)
\]
are t-black-holomorphic. Since these functions are uniformly bounded, they are equicontinuous and hence so are~$\FF{+\pm}(\,\cdot\,,\wc)$ and their conjugate~$\FF{-\mp}$; moreover, the estimate on the modulus of continuity depends only on the maximum of these functions and hence is uniform in~$\wc$. A similar argument applies to the second coordinate; together this implies that functions~$\FF{\pm\pm}$ are equicontinuous on compact subsets of~$U_1\times U_2$ and thus form a pre-compact family. The property (i) of subsequential limits inherits the same property for~$\FF{\pm\pm}$ and the properties (ii),(iii) follow from~\eqref{eq:closed-forms}.
\end{proof}

\begin{theorem} \label{thm:F-via-H}
Let~$F^\delta_\frb$, $\delta={\delta_m}\to 0$, be a sequence of t-black- (or, similarly, t-white-) holomorphic functions defined on t-embeddings~$\cT^\delta$ that contain a fixed open set~$U\subset\C$ and satisfy both assumptions~{\LipKd} and~{\ExpFat} on compacts $K\subset U$. Assume that for each~$\delta$ there exists an~$\alpha\in\mathbb{T}$ such that the oscillations~$\osc_K \rI_{\alpha\R}[F^\delta_\frb]$ are uniformly bounded for each compact $K\subset U$. Then, the functions~$F^\delta_\frb$ are uniformly bounded on compact subsets of~$U$.
\end{theorem}
\begin{proof}
See \cite[Theorem~6.17]{CLR1} and the discussion right above it.
\end{proof}
\begin{corollary} \label{cor:F-via-H}
In the setup of Theorem~\ref{thm:F-via-H} the family~$\{F^\delta_\frb\}$ is precompact in the topology of the uniform convergence on compact subsets~$K\subset U$.
\end{corollary}
\begin{proof}
This is a simple combination of Theorem~\ref{thm:F-precomp} and Theorem~\ref{thm:F-via-H}.
\end{proof}


\newcommand{\tlim}{z}
\newcommand{\olim}{\vartheta}
\newcommand{\w}{\zeta}

\newcommand\z[1]{z^{{\scriptscriptstyle [}#1{\scriptscriptstyle{]}}}}
\newcommand\oo[1]{\vartheta^{{\scriptscriptstyle [}#1{\scriptscriptstyle{]}}}}
\newcommand\zzeta[1]{\zeta^{{\scriptscriptstyle [}#1{\scriptscriptstyle{]}}}}

\newcommand\zz[2]{z_{{\scriptscriptstyle [}#1{\scriptscriptstyle{]}}}^{{\scriptscriptstyle [}#2{\scriptscriptstyle{]}}}}

\renewcommand\aa[2]{\omega_{{\scriptscriptstyle [}#1{\scriptscriptstyle ]}}^{{\scriptscriptstyle [}#2{\scriptscriptstyle ]}}}
\newcommand\bb[2]{\beta_{{\scriptscriptstyle [}#1{\scriptscriptstyle]}}^{{\scriptscriptstyle [}#2{\scriptscriptstyle]}}}

\newcommand\ppsi[1]{\psi^{{\scriptscriptstyle [}#1{\scriptscriptstyle ]}}}

\newcommand\ww[2]{w_{\scriptscriptstyle [#1]}^{\scriptscriptstyle [#2]}}
\newcommand\pzpz[3]{\frac{\partial\zz{#1}{#2}}{\partial\zzeta{#3}}}

\subsection{Limits of t-holomorphic functions and holomorphicity on~$\rS_\xi$}\label{sub:hol-on-S}

Recall that~$\zeta\in\mathbb{D}$ stands for a (orientation-preserving) conformal parametrization of the space-like \emph{Lorentz-minimal} surface~$\rS_\xi$, which means that this parametrization is both conformal and \emph{harmonic}, i.e., {that it satisfies the equations~\eqref{eq:conf-param}.} Though in the actual setup of Theorem~\ref{thm:main-GFF} the limit~$\olim$ of the origami maps~$\cO^\delta$ is necessarily a real-valued function, we do \emph{not} assume this in what follows, in particular, keeping in mind more general applications of our techniques. The condition that the surface~$\rS_\xi$ is space-like reads as
\[
|\,\overline{\alpha}\cdot\pazeta\olim +\alpha\cdot\opazeta\olim\,|\ <\ |\,\overline{\alpha}\cdot\pazeta\tlim+\alpha\cdot\opazeta\tlim\,|\ \ \text{for all~$\alpha\in\mathbb{T}$\,;}
\]
in particular, we have
\begin{equation}
\label{eq:z-zeta>theta-zeta}
|\pazeta\tlim|>|\pazeta\olim|\ge|\opazeta\tlim|\quad \text{and}\quad
|\pazeta\tlim|>|\opazeta\olim|\ge|\opazeta\tlim|\,.
\end{equation}
Moreover, the mapping~$\zeta\mapsto z(\zeta)\in\Omega_\xi$ is quasi-conformal on compact subsets of the unit disc and is proper, i.e., one has~$z(\zeta)\to\partial\Omega_\xi$ if~$\zeta\to\partial\mathbb{D}$.

In what follows we often use the following shorthand notation for the coordinates in the Minkowski space~$\R^{2+2}$:
\[
\z\ssp=\zz\ssp\ssp:=z,\ \ \z\ssm=\zz\ssp\ssm:=\overline{z},\ \ \oo\ssp=\zz\ssm\ssp:=\vartheta,\ \ \oo\ssm=\zz\ssm\ssm:=\overline{\vartheta}
\]
and a similar notation~$\zzeta\ssp:=\zeta$, $\zzeta\ssm:=\overline{\zeta}$ for the conformal coordinate on~$\rS_\xi$. Also, denote
\begin{equation}
\label{eq:bb-def}
\begin{array}{ll}
\bb\ssp\ssp(\zeta):=(\pazeta\tlim)^{1/2}, \quad  & \bb\ssm\ssp(\zeta):={\pazeta\olim}\cdot{(\pazeta\tlim)^{-1/2}},\\[4pt]
\bb\ssp\ssm(\zeta):={\opazeta\overline\olim}\cdot{(\opazeta\overline z)^{-1/2}},\quad & \bb\ssm\ssm(\zeta):=(\opazeta\overline{z})^{1/2} \end{array}
\end{equation}
and
\begin{equation}
\label{eq:aa-def}
\begin{array}{ll}
\aa\ssp\ssp(\zeta):=(\pazeta\tlim)^{1/2}, \quad & \aa\ssm\ssp(\zeta):={\pazeta\overline\olim}\cdot {(\pazeta\tlim)^{-1/2}},\\[4pt]
\aa\ssp\ssm(\zeta):={\opazeta\olim}\cdot{{(\opazeta\overline \tlim})^{-1/2}},\quad & \aa\ssm\ssm(\zeta):=(\opazeta\overline{\tlim})^{1/2};
\end{array}
\end{equation}
see the forthcoming Proposition~\ref{prop:hol-in-zeta}, Eq.~\eqref{eq:psipmpm-def} and Proposition~\ref{prop:psipmpm-def} for the motivation of this notation. Since the parametrization~$\zeta$ is conformal, for all~$p,q\in\{\pm\}$ one has
\begin{equation}
\label{eq:dzz=bbaadzzeta}
d\zz{pq}{q}\ =\ \bb{p}{\ssp}\aa{q}{\ssp}d\zeta\ +\ \bb{p}{\ssm}\aa{q}{\ssm}d\overline{\zeta}\ =\ \textstyle \sum_{r\in\{\pm\}}\bb{p}{r}\aa{q}{r}d\zzeta{r}\,.
\end{equation}
Note also that the functions~$\bb\sspm\ssp$ and~$\aa\sspm\ssp$ are holomorphic in~$\zeta$ while their conjugate~$\bb\ssmp\ssm$ and~$\aa\ssmp\ssm$ are anti-holomorphic since~$\opazeta\pazeta\tlim=\opazeta\pazeta\olim=0$.

The following proposition re-interprets the condition~\eqref{eq:closed-forms} as the usual holomorphicity property in the conformal parametrization of the surface~$\rS_\xi$.
\begin{proposition} \label{prop:hol-in-zeta}
In the setup of Theorem~\ref{thm:main-GFF}, let $\fb,\fw:U\to\C$ be continuous functions defined on an open set $U\subset\Omega_\xi$. Denote
\begin{align*}
\psi_\frb(\zeta)\ &:=\ \bb\ssp\ssp(\zeta)\cdot \fb(z(\zeta))+\bb\ssm\ssp(\zeta)\cdot \overline{\fb(z(\zeta))},\\
\psi_\frw(\zeta)\ &:=\ \aa\ssp\ssp(\zeta)\cdot \fw(z(\zeta))+\aa\ssm\ssp(\zeta)\cdot \overline{\fw(z(\zeta))}.
\end{align*}
Then, the differential form~$\fb dz+\ofb d\vartheta$ is closed if and only if the function~$\psi_\frb$ is holomorphic in~$\zeta$.
Similarly, the differential form~$\fw dz+\ofw d\overline\vartheta$ is closed if and only if the function~$\psi_\frw$ is holomorphic in~$\zeta$. Moreover,
\begin{equation}
\label{eq:fbfwdz=pbpwdzeta}
\Re[\,\fb\fw dz+\ofb\fw d\olim\,]\ =\ \Re[\,\psi_\frb\psi_\frw d\zeta\,]\,.
\end{equation}
In particular, the differential form~\eqref{eq:fbfwdz=pbpwdzeta} is closed and its primitive is harmonic in the conformal metric of the surface~$\rS_\xi$.
\end{proposition}
\begin{remark} Note that~$|\bb\ssp\ssp(\zeta)|>|\bb\ssm\ssp(\zeta)|$ for all~$\zeta\in \mathbb{D}$ due to~\eqref{eq:z-zeta>theta-zeta}. Therefore, one can uniquely reconstruct the function~$\fb$ from~$\psi_\frb$ (a similar statement holds for the functions $\fw$ and $\psi_\frw$).
\end{remark}
\begin{proof}[Proof of Proposition~\ref{prop:hol-in-zeta}]
Let~$V:=\zeta(U)\subset\mathbb{D}$ be the image of the open set~$U$ under the mapping~$z\mapsto \zeta(z)$. Consider the primitive
\[
F(\zeta):=\int^{z(\zeta)}\big(\fb(z)dz+\overline{\fb(z)}d\vartheta(z)\big),\quad \zeta\in V.
\]
It is easy to see that~$F$ satisfies the conjugate Beltrami equation
\[
\opazeta F(\zeta)\,=\,
\nu(\zeta)\cdot\overline{\pazeta F(\zeta)}\,,\quad \text{where}\quad \nu\,:=\,\opazeta\olim/ \opazeta\overline{\tlim}=\opazeta\tlim/\opazeta\overline{\olim}\,.
\]
Note that~$\pazeta\nu=0$ and let~$\phi\in C_0^\infty(V)$ be a test function. Since
\begin{align*}
\iint F\cdot \opazeta\pazeta\phi\ &=\ -\iint \partial_{\bar\zeta}F\cdot \partial_\zeta\phi\ =\ -\iint \partial_{\bar\zeta}\overline{F}\cdot \partial_\zeta(\nu\phi)\\
&=\ \iint \overline{F}\cdot \opazeta\pazeta(\nu\phi)\ =\ ...\ =\ \iint F\cdot\opazeta\pazeta(|\nu|^2\phi)
\end{align*}
we see that the function~$F$ is harmonic (note that~$|\nu(\zeta)|<1$ for all~$\zeta\in\mathbb{D}$ as the surface~$\rS_\xi$ is space-like). This implies that the derivative
\[
\pazeta F\ =\ \fb\cdot \pazeta z+\ofb\cdot \pazeta\vartheta\ =\ \psi_\frb\cdot \bb\ssp\ssp
\]
is a holomorphic function of~$\zeta$. Therefore, the function~$\psi_\frb$ itself is holomorphic since so is~$\bb\ssp\ssp$.

Vice versa, if the function~$\psi_\frb$ is holomorphic, then so are the products~$\psi_\frb\bb\sspm\ssp$ and hence the differential form
$\fb dz+\ofb d\vartheta=\psi_\frb\bb\ssp\ssp d\zeta+\overline{\psi}_\frb\bb\ssp\ssm d\overline{\zeta}$ is closed.

 The proof for~$\psi_\frw$ is similar. The identity~\eqref{eq:fbfwdz=pbpwdzeta} directly follows from the definitions~\eqref{eq:bb-def},~\eqref{eq:aa-def} of the coefficients~$\bb\sspm\sspm$, $\aa\sspm\sspm$.
\end{proof}

Proposition~\ref{prop:hol-in-zeta} applies to all (subsequential) limits of t-holomorphic functions on t-embeddings~$\cT^\delta$. We now prove a similar result for subsequential limits~$\ff{\sspm\sspm}$ of functions~$\FF{\pm\pm}_{\cT^\delta}$; see Corollary~\ref{cor:Fpmpm-hol-in-zeta}. For $r_1,r_2\in\{\pm\}$ denote
\begin{equation}
\label{eq:psipmpm-def}
\ppsi{r_1,r_2}(\zeta_1,\zeta_2)\ := \sum\nolimits_{s_1,s_2\in\{\pm\}}\bb{s_1}{r_1}(\zeta_1)\aa{s_2}{r_2}(\zeta_2)\ff{s_1,s_2}(z(\zeta_1),z(\zeta_2))\,,
\end{equation}
where $\bb\sspm\sspm(\zeta_1)$ and~$\aa\sspm\sspm(\zeta_2)$ are given by~\eqref{eq:bb-def} and~\eqref{eq:aa-def}; note that
\[
\ppsi{\ssm\ssm}(\zeta_1,\zeta_2)=\overline{\ppsi{\ssp\ssp}(\zeta_1,\zeta_2)}\quad \text{and}\quad \ppsi{\ssp\ssm}(\zeta_1,\zeta_2)=\overline{\ppsi{\ssm\ssp}(\zeta_1,\zeta_2)}
\]
due to the similar symmetries of the functions~$\ff{\sspm\sspm}(\zeta_1,\zeta_2)$, $\bb\sspm\sspm(\zeta_1)$, $\aa\sspm\sspm(\zeta_2)$.
\begin{proposition}
\label{prop:psipmpm-def}
In the setup of Theorem~\ref{thm:main-GFF}, let~$U_1,U_2$ be disjoint open subsets of~$\Omega_\xi$. Assume that the complexified dimer coupling functions $\FF{\pm\pm}_{\cT^\delta}$ converge, as $\delta\to 0$, to continuous functions~$\ff{\sspm\sspm}:U_1\times U_2\to\C$ uniformly on compact subsets of~$U_1\times U_2$. Let~$V_1:=\zeta(U_1)$ and~$V_2:=\zeta(U_2)$. Then, \\
(i)\phantom{i} for each~$\zeta_2\in V_2$ both functions $\ppsi{\ssp\sspm}(\,\cdot\,,\zeta_2)$ are holomorphic in $V_1$;\\
(ii) for each~$\zeta_1\in V_1$ both functions $\ppsi{\sspm\ssp}(\zeta_1,\,\cdot\,)$ are holomorphic in $V_2$.
\end{proposition}
\begin{proof} 
It follows from Proposition~\ref{prop:Fpmpm-def}(iii) and Proposition~\ref{prop:hol-in-zeta} that for each $\eta\in\C$ the function
\begin{align*}
\bb\ssp\ssp(\cdot)\big(&\overline{\eta}\ff{\ssp\ssp}(z(\cdot),z(\zeta_2))+\eta\ff{\ssp\ssm}(z(\cdot),z(\zeta_2))\big)\\
&+\bb\ssm\ssp(\cdot)\big(\eta\ff{\ssm\ssm}(z(\cdot),z(\zeta_2))+\overline{\eta}\ff{\ssm\ssp}(z(\cdot),z(\zeta_2))\big).
\end{align*}
is holomorphic in~$V_1$. Varying~$\eta$ one sees that this is only possible if so are both functions
\[
\bb\ssp\ssp(\cdot)\ff{\ssp,s}(z(\cdot),z(\zeta_2))+\bb\ssm\ssp(\cdot)\ff{\ssm,s}(z(\cdot),z(\zeta_2)),\quad s\in\{\pm\}.
\]
Taking their linear combination with the coefficients~$\aa{s}{r}(\zeta_2)$ we obtain the function~$\ppsi{\ssp,r}(\cdot,\zeta_2)$, which also has to be holomorphic in~$V_1$ for both~$r\in\{\pm\}$. The proof of (ii) is similar.
\end{proof}

\section{Proof of Theorem~\ref{thm:main-GFF}}\label{sec:main-proof}
\subsection{Uniform boundedness of the inverse Kasteleyn matrix}\label{sub:boundedness}
We now prove the key estimate required to apply the general scheme developed in~\cite[Theorem~1.4]{CLR1} in the setup of this paper.
\begin{theorem}\label{thm:K-1=O(1)}
In the setup of Theorem~\ref{thm:main-GFF}, the functions~$K^{-1}_{\cT^\delta}(w,b)$ are uniformly bounded as~$\delta\to 0$ provided that $w$ and~$b$ stay at a definite distance from each other and at least one of them stays at a definite distance from~$\partial\Omega_\xi$.

Moreover, the functions~$\FF{\pm\pm}_{\cT^\delta}(\bc,\wc)$ defined in Proposition~\ref{prop:Fpmpm-def} are also uniformly bounded as~$\delta\to 0$ provided that~$\bc,\wc$ stay at a definite distance from each other and at least one of them stays at a definite distance from~$\partial\Omega_\xi$.
\end{theorem}

\begin{proof}
Let $d_0>0$ be fixed and assume, e.g., that~$\mathrm{dist}(w,\partial\Omega_\xi)\ge 4d_0$; the other case~$\mathrm{dist}(b,\partial\Omega_\xi)\ge 4d_0$ is fully symmetric. Below we prove an estimate $F_w^\tw(\wc)=O(1)$ for the `true' complex values~$F_w^\tw$ of t-holomorphic functions~$F_w$; recall that~$F_w^\tb(\,\cdot\,)=\overline{\eta}_w K^{-1}(w,\,\cdot\,)$. Let us emphasize that, here and below, the implicit constants in such estimates are allowed to depend on~$d_0$. The proof goes in two steps.
\begin{itemize}
\item \textbf{Step~1.} We prove that~$F_w^\tw(\wc)=O(1)$ for the points~${\wc=w^\pm_{\mathrm{out},k}}$ lying at the \emph{outer boundary} of~$\cT^\delta$; see Fig.~\ref{fig:p-emb}.

\smallskip

\item\textbf{Step~2.} We prove that the functions~$F_w^\tw$ remain uniformly bounded as~$\delta\to 0$ at points lying at a fixed (small) distance~$d_0$ from~$w$.
\end{itemize}
The proofs of these two steps are given below. Then, the uniform estimate
\begin{equation}
\label{eq:Fw=O(1)}
F_w^\circ(\wc)=O(1)\ \ \text{for all~$\wc$ such that $|\wc-w|\ge d_0$}
\end{equation}
follows from the maximum principle applied to the \mbox{t-holomorphic} function~$F_w^\circ$; see Lemma~\ref{lem:max-t-hol}. Since~$K^{-1}(w,b)=\eta_w F^\tb_w(b)=\Pr(F^\tw_w(\wc),\eta_b\R)$ for~$\wc\sim b$, this immediately gives the desired~$O(1)$ bound for the entries of the inverse Kasteleyn matrices~$K^{-1}_{\cT^\delta}$. Recall that
\[
F_w^\circ(\wc)\ =\ \tfrac12(\overline{\eta}_w \FF{++}(\bc,\wc)+\eta_w \FF{-+}(\bc,\wc))\ \ \text{if}\ \ \bc\sim w.
\]
As t-embeddings~$\cT^\delta$ satisfies assumptions~\LipKd\ and~\ExpFat\ \emph{in the bulk} of~$\Omega_\xi$, the uniform estimate~\eqref{eq:Fw=O(1)} also implies a (formally, stronger) estimate
\[
\FF{\pm\pm}_{\cT^\delta}(\bc,\wc)=O(1)\ \ \text{if}\ \ \mathrm{dist}(\bc,\partial\Omega_\xi)\ge 4d_0\ \ \text{and}\ \ |\wc-\bc|\ge d_0;
\]
see the proof of~\cite[Proposition~6.21]{CLR1}.
\end{proof}

We now pass to the proofs of Step~1 and Step~2. The following lemma is the key ingredient of Step~1.

\begin{lemma}\label{lem:same-sign}
Let~$b\in\partial^\tb\cT^\delta$ be a boundary face of a perfect t-embedding~$\cT$. Then, one can find a direction~$\lambda=\lambda_b\in\mathbb{T}$ such that all increments of the real-valued function~$\overline{\lambda}\rI_{\lambda\R}[F_b]=2\Re(\overline{\lambda}\rI_\C[F_b])$ along the outer boundary of~$\cT$ are of the same sign and sum up to~$\Re[\lambda\eta_b]$.

This implies that~$|\rI_{\lambda\R}[F_b]|\le 1$ everywhere on~$\cT$ provided that the additive constant in the definition of the primitive $\rI_{\lambda\R}[F_b]$ is chosen appropriately.
\end{lemma}
\begin{proof} For concreteness, assume that~$b=b_n$. Since~$F_b^\tw(w)=\overline{\eta}_b K^{-1}(w,b)$, the primitive~$\rI_\C[F_b]$ has no jumps along black boundary edges $(v_{2k}v_{2k+1})$ of~$\cT$ for all~$k\le n-1$ while $\rI_\C[F_b](v_1)-\rI_\C[F_b](v_{2n})=2\overline{\eta}_b$. (This quantity is nothing but the additive monodromy of the primitive~$\rI_\C[F_b]$ around~$b_n$.)

Let us now find the signs of~$F_b^\circ(w_k)=\overline{\eta}_b K^{-1}(w_k,b)$ for~$k=1,\ldots,n$. To this end, let us add to the graph~$\mathcal{G}$ an additional edge~$e_\mathrm{ext}$ of a (small) weight~$\varepsilon>0$ going from~$b=b_n$ to~$w_k$ counterclockwise inside the outer face of~$\cT$. It is easy to see that the Kasteleyn condition on the extended graph~$G_\mathrm{ext}:=G\cup e_\mathrm{ext}$ holds if (and only if) we give a complex sign
\[
K_\mathrm{ext}(e_\mathrm{ext})\ :=\ \varepsilon\cdot \textstyle \prod_{j=1}^{k}e^{i\phi_{2j-1}}\prod_{j=1}^{k-1}e^{-i\phi_{2j}}
\]
to this newly drawn edge; we set~$K_\mathrm{ext}(e):=K(e)$ for all other edges of~$\mathcal{G}$. The Kasteleyn theorem implies that
\[
K_\mathrm{ext}^{-1}(w_k,b_n)K_\mathrm{ext}(b_n,w_k)\in \R_+
\]
(this is the probability that~$e_\mathrm{ext}$ belongs to a random dimer covering of~$\mathcal{G}_\mathrm{ext}$). Since~$K_\mathrm{ext}^{-1}(w_k,b_n)\to K^{-1}(w_k,b_n)$ as~$\varepsilon\to 0$ we see that
\[
F_b^\tw(w_k)=\overline{\eta}_bK^{-1}(w_k,b_n)\ \in\ \textstyle \overline{\eta}_b\prod_{j=1}^{k}e^{-i\phi_{2j-1}}\prod_{j=1}^{k-1}e^{i\phi_{2j}}\cdot \R_+\,.
\]
Recall that~$\cT(v_{2k})-\cT(v_{2k-1})\in \overline{\eta}{}_{w_k}^2\R_+ =ie^{i(\phi_{2k-1}-\xi_{2k-1})}\R_+$. Therefore,
\begin{align*}
\rI_\C[F_b](v_{2k})-\rI_\C[F_b](v_{2k-1})\ &\in\ i\overline{\eta}_b e^{-i\xi_{2k-1}+\sum_{j=1}^{k-1}(\phi_{2j}-\phi_{2j-1})}\cdot \R_+\\
&=\ \textstyle i\overline{\eta}_b e^{-i\xi_1-\sum_{j=1}^{k-1}(\phi_{2j}-\phi_{2j+1})}\cdot \R_+
\end{align*}
since~$\phi_{2j}-\phi_{2j-1}=\xi_{2j}-\xi_{2j-1}$ and $\phi_{2j}-\phi_{2j+1}=\xi_{2j+1}-\xi_{2j}$. It remains to choose, e.g.,~$\lambda:=\overline{\eta}_b e^{-i\xi_1}$ and note that $\sum_{j=1}^{k-1}(\phi_{2j}-\phi_{2j-1})\in [0,\pi]$ since all the terms in this sum are positive and~$\sum_{j=1}^n(\phi_{2j}-\phi_{2j-1})=\pi$.

To conclude, recall that $\rI_{\lambda\R}[F_b](v_{2k})=\rI_{\lambda\R}[F_b](v_{2k+1})$ for all~$k\le n-1$. Since all the terms in the sum
\begin{align*}
\textstyle \sum_{k=1}^{2n}(\rI_{\lambda\R}[F_b](v_{2k})-\rI_{\lambda\R}[F_b](v_{2k-1}))\ &=\ \rI_{\lambda\R}[F_b](v_{2n})-\rI_{\lambda\R}[F_b](v_1)\\
&=\ 2\lambda\Re(\lambda\eta_b).
\end{align*}
have the same sign, one can choose an additive constant in the definition of the primitive~$\rI_{\lambda\R}[F_b]$ so that~$|\rI_{\lambda\R}[F_b]|\le 1$ for all outer vertices. This estimate trivially extends to the bulk of~$\cT$ due to the maximum principle for harmonic functions on T-graphs.
\end{proof}

\begin{proof}[Proof of Step~1] Let a white face~$w$ of~$\cT^\delta$ be such that~$\mathrm{dist}(w,\partial\Omega_\xi)\ge 4d_0$. It follows from Lemma~\ref{lem:same-sign} and the a priori regularity theory for t-holomorphic functions (see Theorem~\ref{thm:F-via-H}) that
\[
F_w^\tb(b)=\overline{\eta}_w K^{-1}(w,b)=\overline{\eta}_w\eta_b F_b^\circ(w)\ =\ O(1)\ \ \text{as}\ \ \delta\to 0,
\]
uniformly over \emph{boundary} faces~$b\in\partial^\bullet\cT^\delta$ (and over \emph{inner} faces~$w$ lying at a definite distance from~$\partial\Omega_\xi$). This implies the required uniform estimate
\begin{equation}
\label{eq:step-1}
F_w^\tw(\wc)\ =\ O(1)\ \ \text{as}\ \ \delta\to 0,\ \ \text{uniformly over}\ \wc\in\partial^\tw_\mathrm{out}\cT^\delta,
\end{equation}
since {we have, similarly to~\eqref{eq:F-pa-out-def},}
\[
{\Pr(F^\tw_w(w_{\mathrm{out},k}^\pm),\eta_{b_k}\R)=F^\tb_w(b_k)\,,\quad
\begin{array}{l}
\Pr(F^\tw_w(w_{\mathrm{out},k}^+),\eta_{b_{\mathrm{out},k+1}}\R) =0\,,\\[4pt]
\Pr(F^\tw_w(w_{\mathrm{out},k}^-),\eta_{b_{\mathrm{out},k}}\R) =0\,,
\end{array}}
\]
and the angles between the directions~$\eta_{b_k}\R$ and~$\eta_{b_{\mathrm{out},k}}\R$ or~$\eta_{b_{\mathrm{out},{k+1}}}\R$ remain uniformly isolated from~$0$ and~$\pi$; see Remark~\ref{rem:xi-k}.
\end{proof}
\begin{remark} Let us emphasize that the estimate~\eqref{eq:step-1} is \emph{not} uniform in~$d_0$. Indeed, estimating t-holomorphic function~$F_w$ via Theorem~\ref{thm:F-via-H} already costs a $O(d_0^{-1})$ factor. This dependence on~$d_0$ can be even worse since {these} a priori regularity estimates also implicitly depend on the constant~$\kappa$ in the assumption~{\LipKd} and it can happen that~$\kappa\to 1$ as~$d_0\to 0$.
\end{remark}

We now move to Step~2. Let~$w=w^\delta$ be an inner white face of~$\cT^\delta$ and denote~$\eta^\delta:=\eta_{w^\delta}$ for shortness.
Consider the (complex-valued) primitive $\rI_\C[F^\delta_{w^\delta}]$ of the t-holomorphic function~$F_{w^\delta}$; recall that~$F_w^\bullet(b)=\overline{\eta}_w K^{-1}(w,b)$. Strictly speaking, $\rI_\C[F_{w^\delta}]$ is not a well-defined function on~$\cT^\delta$ since it has an additive monodromy~$2\int_{\partial w^\delta}F_{w^\delta}^\bullet=2\overline{\eta}{}^\delta$ around~$w^\delta$. However, the (real-valued) function
\begin{equation}
\label{eq:H-delta-def}
H^\delta\ :=\ -i\eta^\delta\rI_{i\overline{\eta}{}^\delta\R}[F_{w^\delta}]
\end{equation}
has \emph{zero} monodromy around~$w^\delta$ and hence is well defined on~$\cT^\delta$, up to a global additive constant. Moreover, the face~$w^\delta$ becomes a degenerate vertex in the T-graph~$\cT^\delta+(i\overline{\eta}{}^\delta)^2\cO^\delta$ and the function~$H^\delta$ is harmonic on this \mbox{T-graph} everywhere except this vertex. Replacing~$H^\delta$ by~$-H^\delta$ if needed we can also assume that~$H^\delta$ is super-harmonic at~$w^\delta$ (i.e., satisfies the minimum principle in a vicinity of~$w^\delta$).

Let~$\Gamma^\delta$ be a contour on~$\cT^\delta$ approximating the circle
\[
\Gamma:=\{z:|z-w|=2d_0\}.
\]
We now consider a (unique) decomposition
\begin{equation}
\label{eq:H=G+H0}
H^\delta\ =\ G^\delta+H^\delta_\mathrm{harm}\ \ \text{inside}\ \ \Gamma^\delta
\end{equation}
of the restriction of~$H^\delta$ onto the interior of~$\Gamma^\delta$ (unioned with an appropriate subset of $\Gamma^\delta$ consisting of all boundary vertices of the image of this interior in the corresponding T-graph~$\cT^\delta+(i\overline{\eta}{}^\delta)^2\cO^\delta$), where
\begin{itemize}
\item the function~$G^\delta$ has zero boundary values on~$\Gamma^\delta$ and is discrete harmonic in the interior of~$\Gamma^\delta$ except at the point~$w^\delta$;
\item the function~$H^\delta_\mathrm{harm}$ is discrete harmonic everywhere inside~$\Gamma^\delta$.
\end{itemize}
Let us also consider a t-white-holomorphic function
\[
F^\delta\ :=\ F_{w^\delta}-\rD[i\overline{\eta}{}^\delta H^\delta_\mathrm{harm}],
\]
see~Proposition~\ref{prop:t-hol=grad}. Note that $\rI_{i\overline\eta{}^\delta\R}[F^\delta]=H^\delta-H^\delta_\mathrm{harm}=G^\delta$ and it is easy to see that the complex-valued primitive~$\rI_\C[F^\delta]$ has the same additive monodromy $2\overline{\eta}{}^\delta$ around $w^\delta$ as~$\rI_\C[F_{w^\delta}]$. We start with a preliminary lemma.
\begin{lemma}\label{lem:conv-to-Green} In the setup of Theorem~\ref{thm:main-GFF}, the functions~$\frac{1}{2}G^\delta$ converge, as $\delta\to 0$, to the Dirichlet Green's function in~$B(w,2d_0)$, where the harmonicity is understood in the metric of the Lorentz-minimal surface~$\rS_\xi$. The convergence is uniform on compact subsets of~$B(w,2d_0)\smallsetminus\{w\}$.
\end{lemma}

\begin{remark} In what follows we rely upon the fact that the assumption $\LipKd$\ holds {up to the boundary} of the auxiliary domain~$B(w,2d_0)$. We are not aware of a similar argument that could be used directly in the domain~$\Omega_\xi$.
\end{remark}

\begin{proof}[Proof of Lemma~\ref{lem:conv-to-Green}] Without loss of generality we can assume that~$\eta^\delta\to\eta$ as~$\delta\to 0$ due to compactness arguments. Let~$v_0$ be a fixed `reference' point inside the disc~$B(w,2d_0)$ lying, say, at distance~$d_0$ from~$w$. Recall that the functions~$G^\delta$ are discrete super-harmonic (on the corresponding T-graphs) except at~$w^\delta$ and thus are non-negative. We now consider two cases:

\smallskip

\noindent \emph{Case~A (`normal'). The values~$G^\delta(v_0^\delta)$ remain bounded as~$\delta\to 0$.} In this case the discrete Harnack principle (see Proposition~\ref{prop:Harnack}) implies that the functions~$G^\delta$ are also uniformly bounded on all compact subsets of the punctured disc~$B(w,2d_0)\smallsetminus\{w\}$. Applying the a priori regularity theory (see Theorem~\ref{thm:F-via-H} and Corollary~\ref{cor:F-via-H}) to \mbox{t-holomorphic} functions $F^\delta=\rD[G^\delta]$  one can find a subsequential limit
\[
F^{\tw,\delta}\ \rightrightarrows\ \fw:B(w,2d_0)\smallsetminus\{w\}\to \C\ \ \text{as}\ \ \delta\to 0,
\]
where the convergence is uniform on compact subsets. This also implies that
\[
G^\delta(\cdot)\ \rightrightarrows\ g_\R(\cdot):=\pm\Re[i\eta g_\C(\cdot)]\,,\quad g_\C(v)=\pm\int^v (\fw(z)dz+\overline{\fw(z)}d\overline{\theta(z)}\,)\,,
\]
also uniformly on compact subsets of the punctured disc~$B(w,2d_0)\smallsetminus\{w\}$. Due to Proposition~\ref{prop:hol-in-zeta} the function~$g_\C$ is harmonic in the metric of the surface~$\rS_\xi$. Since the functions~$H^\delta$ are non-negative, so is their limit~$g_\R$. Moreover, the uniform crossing estimates for random walks on T-graphs developed in~\cite[Section~6]{CLR1} imply that~$g_\R$ has zero boundary values on~$\Gamma$ since
\[
H^\delta(v)\to 0\ \ \text{as}\ \ v\to\Gamma,\ \ \text{uniformly in~$\delta$}.
\]
(Let us emphasize that here we rely upon the fact that the assumption \LipKd\ holds \emph{up to the boundary} of the disc~$B(w,2d_0)$ and not only on its compact subsets.) Therefore, $g_\R$ must be a multiple of the Green function in~$B(w,2d_0)$. It remains to note that the additive monodromy of the complex-valued function~$g_\C$ around~$w$ equals~$2$ since this monodromy can be computed along the circle $\{z:|z-w|=d_0\}$, on which we have the convergence~$F^{\circ,\delta}\to \fw$. Therefore, $\frac12 g_\R$ is the Green function in~$B(w,2d_0)$.

\smallskip

\noindent \emph{Case B (`pathological'). One has~$G^\delta(v_0^\delta)\to +\infty$ for a sequence \mbox{$\delta={\delta_m}\to 0$}.} In order to rule out this scenario let us consider renormalized functions
\[
\widetilde{G}^\delta(\cdot)\ :=\ G^\delta(\cdot)/G^\delta(v_0^\delta).
\]
By definition, the values~$\widetilde{G}^\delta(v_0^\delta)$ are bounded and hence one can apply the arguments from Case~A to these functions. Since~$G^\delta(v_0^\delta)\to +\infty$, the limiting function~$\widetilde{h}_\C$ has \emph{no} additive monodromy around~$w$, which means that
\[
\widetilde{G}^\delta\ \rightrightarrows\ \widetilde{h}_\R=0
\]
uniformly on compact subsets of the punctured disc~$B(w,2d_0)\smallsetminus\{w\}$. However, such a convergence contradicts to the normalization~$\widetilde{G}^\delta(v^\delta_0)=1$.
\end{proof}

We are now in the position to complete Step~2 in the proof of Theorem~\ref{thm:K-1=O(1)}.

\begin{proof}[Proof of Step~2] Recall that the functions~$H^\delta$ are defined (up to global additive constants) by~\eqref{eq:H-delta-def} and note that
\begin{equation}
\label{eq:osc-H=O(1)}
\osc_{\partial\cT^\delta}(H^\delta)\ =\ O(1)\ \ \text{as}\ \ \delta\to 0
\end{equation}
due to the uniform estimate~$F_{w^\delta}=O(1)$ on~$\partial\cT^\delta$ that was obtained on Step~1 and since the lengths of the boundaries of perfect t-embeddings~$\cT^\delta$ converging to the domain~$\Omega_\xi$ are bounded uniformly in~$\delta$. (Let us mention once again that these~$O(1)$ constants may depend on~$d_0$.)

Recall that the discrete contour~$\Gamma^\delta$ approximates the boundary of the disc~$B(w,2d_0)$. Somewhat similarly to the proof of Lemma~\ref{lem:conv-to-Green} we now consider two cases.

\smallskip

\noindent \emph{Case~A. The oscillations~$\osc_{\Gamma^\delta}(H^\delta)$ remain bounded as~$\delta\to 0$.} In this case the desired estimate easily follows from Lemma~\ref{lem:conv-to-Green} via the a priori regularity theory for t-holomorphic functions. Indeed, consider the decomposition~\eqref{eq:H=G+H0} and note that oscillations of the harmonic function~$H^\delta_\mathrm{harm}$ inside the closed contour~$\Gamma^\delta$ are bounded by~$\osc_{\Gamma^\delta}(H^\delta_\mathrm{harm})=\osc_{\Gamma^\delta}(H^\delta)=O(1)$ due to the maximum principle. Together with Lemma~\ref{lem:conv-to-Green} on the convergence of~$G^\delta$ as~$\delta\to 0$ this yields the estimate
\[
\osc_K(H^\delta)=O(1)\ \ \text{as}\ \ \delta\to 0
\]
for each compact set~$K\subset B(w,2d_0)\smallsetminus\{w\}$. Then, Theorem~\ref{thm:F-via-H} implies that the gradients~$F_{w^\delta}=\rD[H^\delta]$ of these harmonic functions are also uniformly bounded on compact subsets of the punctured disc~$B(w,2d_0)\smallsetminus\{w\}$.

\smallskip

\noindent \emph{Case B. One has~$Q^\delta:=\osc_{\Gamma^\delta}(H^\delta)\to+\infty$ as~$\delta={\delta_m}\to 0$.} In order to rule out this `pathological' scenario let us consider renormalized functions
\[
\widetilde{H}^\delta\ :=\ H^\delta/Q^\delta,
\]
where additive constants in the definition of~$H^\delta$ are chosen so that
\begin{itemize}
\item the minimum of~$\widetilde{H}^\delta$ on the contour~$\Gamma^\delta$ equals~$0$, let us denote by~$v_0^\delta$ the point at which this minimum is attained;
\item the maximum of~$\widetilde{H}^\delta$ on~$\Gamma^\delta$, attained at a point~$v_1^\delta$, equals~$1$.
\end{itemize}
Recall that the function~$H^\delta$ is super-harmonic and has uniformly bounded oscillations on~$\partial\cT^\delta$. This implies that~$\widetilde{H}^\delta|_{\partial\cT^\delta}\le  O((Q^\delta)^{-1})=o_{\delta\to 0}(1)$. (However, at least a priori, we do \emph{not} know that~$|\widetilde{H}^\delta|_{\partial\cT^\delta}|=o_{\delta\to 0}(1)$ since it can happen that~$\max_{\Gamma^\delta} H^\delta\gg\osc_{\Gamma^\delta} H^\delta$.)

Due to the maximum principle, there exists a discrete path~$\gamma^\delta$ that goes from~$v_1^\delta$ to (the boundary of) $w^\delta$ such that~$\widetilde{H}^\delta\ge 1$ along~$\gamma^\delta$. (Note that this path could not have ended at~$\partial\cT^\delta$ since~$\widetilde{H}^\delta\le o(1)<1$ everywhere at~$\partial\cT^\delta$.) Consider now a decomposition~$\widetilde{H}^\delta=\widetilde{G}^\delta+\widetilde{H}^\delta_\mathrm{harm}$ inside the contour~$\Gamma^\delta$ similarly to~\eqref{eq:H=G+H0}. It follows from Lemma~\ref{lem:conv-to-Green} and from the assumption~$Q^\delta\to +\infty$ that the first term~$\widetilde{G}^\delta$ converges to zero as~$\delta\to 0$, uniformly on compact subsets of the punctured disc~$B(w,2d_0)\smallsetminus\{w\}$. The second terms~$\widetilde{H}^\delta_\mathrm{harm}$ are uniformly bounded on~$\Gamma^\delta$ and hence inside this contour as well. As above, this allows us to find a subsequential limit~$\widetilde{H}^\delta_\mathrm{harm}\rightrightarrows \widetilde{h}_\mathrm{harm}$ (the convergence is uniform on compact subsets of~$B(w,2d_0)$), which must be a harmonic function in the metric of the surface~$\rS_\xi$ and satisfies the estimates~$0\le\widetilde{h}_\mathrm{harm}\le 1$. Since~$\widetilde{G}^\delta$ vanishes as~$\delta\to 0$ we also have the uniform convergence
\[
\widetilde{H}^\delta\rightrightarrows \widetilde{h}_\mathrm{harm}\ \ \text{on compact subsets of~$B(w,2d_0)\smallsetminus\{w\}$}.
\]
Therefore, the existence of the discrete paths~$\gamma^\delta$ along which one has~$\widetilde{H}^\delta\ge 1$ implies that~$\widetilde{h}_\mathrm{harm}$ is constantly equal to~$1$.

Now note that, for small enough~$\delta$, the discrete harmonic functions~$2-\widetilde{H}^\delta$ have to be positive everywhere on~$\cT^\delta$ outside the ball~$B(w,d_0)$ due to the discrete maximum principle. Moreover, we have $2-\widetilde{H}^\delta(v^\delta_1)=1$. The discrete Harnack principle (see Proposition~\ref{prop:Harnack}) implies that these functions are uniformly bounded on compact subsets of~$\Omega_\xi\smallsetminus B(w,d_0)$ and hence admit subsequential limits
\begin{equation}
\label{eq:x-conv-h-tilde}
2-\widetilde{H}^\delta\ \rightrightarrows\ 2-\widetilde{h}_\mathrm{harm}=1\ \ \text{on compact subsets of~$\Omega_\xi\smallsetminus B(w,d_0)$};
\end{equation}
where the function~$\widetilde{h}_\mathrm{harm}$ is harmonic in the metric of the surface~$\rS_\xi$ and thus is constantly equal to~$1$ since we already know that this function is constant in the annulus~$B(w,2d_0)\smallsetminus B(w,d_0)$.

The uniform convergence~\eqref{eq:x-conv-h-tilde} near the contours~$\Gamma^\delta$ clearly contradicts to the existence of the points~$v^\delta_0$ such that~$\widetilde{H}^\delta(v^\delta_0)=0$. This completes both Step~2 and the proof of Theorem~\ref{thm:K-1=O(1)}.
\end{proof}


\subsection{Convergence to the GFF} \label{sub:convergence}
This section closely follows the proof of~\cite[Theorem~1.4]{CLR1}. It is worth noting that instead of the 'small origami' assumption (i.e. $\cO^\delta\to 0$ as  $\delta\to 0$) used in~\cite{CLR1}, we now assume that the graphs~$(\cT^\delta,\cO^\delta)$ of the origami maps~$\cO^\delta$ converge to a Lorentz-minimal surface~$\cS_\xi\subset\R^{2+1}\subset\R^{2+2}$. However, most of the arguments used in~\cite[Section~7]{CLR1} to show the convergence of the height fluctuations to the GFF in a planar domain~$\Omega$ can be repeated (almost) verbatim in our setup after passing to the conformal parametrization of the surface~$\rS_\xi$ instead of the Euclidean metric in~$\Omega$. The only essential difference is in the analysis of the boundary values of the correlation functions~$H_n^\delta$. As we are unable to prove a stronger version of Theorem~\ref{thm:K-1=O(1)} for \emph{both} points~$w,b$ lying close to~$\partial\cT^\delta$ and do not want to introduce an additional assumption similar to Assumption~(III) in \cite[Theorem~1.4]{CLR1}, we cannot identify the limits of the functions~$H_n^\delta$ themselves even though we identify the limits of their gradients. The latter turns out to be possible due to the fact that in Theorem~\ref{thm:K-1=O(1)} \emph{one} of the points~$w,b$ is allowed to be close to~$\partial\cT^\delta$.

Recall that complexified (in each of the two variables) dimer coupling functions $\FF{\pm\pm}_{\cT^\delta}(\bc,\wc)$ are defined in Proposition~\ref{prop:Fpmpm-def}. Due to Theorem~\ref{thm:K-1=O(1)} these functions are uniformly bounded as~$\delta\to 0$ provided that the points~$\bc,\wc$ stay at a definite distance from each other and at least one of them stays at a definite distance from~$\partial\Omega_\xi$. Therefore, the a priori regularity theory for t-holomorphic functions (see Corollary~\ref{cor:Fpmpm-hol-in-zeta}) implies the existence of \emph{subsequential} limits
\begin{equation}
\label{eq:fpmpm-def}
\FF{\pm\pm}_{\cT^\delta}(\bc,\wc)\,\to\,\ff{\sspm\sspm}(z_1,z_2)\ \ \text{if}\ \ \bc\to z_1,\ \wc\to z_2\ \text{as}\ \delta={\delta_m}\to 0.
\end{equation}
The convergence is uniform provided that the points $z_1,z_2$ remain at a definite distance from each other and from~$\partial\Omega_\xi$. Moreover, the limits~$\ff{\sspm\sspm}$ are uniformly bounded if $z_1,z_2$ stay away from each other and \emph{at least one} of these points stays away from the boundary~$\partial\Omega_\xi$ (see Theorem~\ref{thm:K-1=O(1)}).

Recall that $\zeta\in \mathbb{D}\leftrightarrow (z,\olim)\in\rS_\xi$ is a conformal parametrization of the surface~$\rS_\xi$ (see~\eqref{eq:conf-param}) and define
\[
\textstyle \ppsi{r_1,r_2}(\zeta_1,\zeta_2)\ :=\ \sum_{s_1,s_2\in\{\pm\}}\bb{s_1}{r_1}(\zeta_1)\aa{s_2}{r_2}(\zeta_2)\ff{s_1,s_2}(z(\zeta_1),z(\zeta_2))
\]
for~$r_1,r_2\in\{\pm\}$, where~$\bb\sspm\sspm(\zeta_1)$ and~$\aa\sspm\sspm(\zeta_2)$ are given by~\eqref{eq:bb-def} and~\eqref{eq:aa-def}. It follows from Proposition~\ref{prop:psipmpm-def} that
\begin{itemize}
\item $\ppsi{\ssm\ssm}(\zeta_1,\zeta_2)=\overline{\ppsi{\ssp\ssp}(\zeta_1,\zeta_2)}$ and $\ppsi{\ssp\ssm}(\zeta_1,\zeta_2)=\overline{\ppsi{\ssm\ssp}(\zeta_1,\zeta_2)}$\,;
\item the functions $\ppsi{\sspm\ssp}(\zeta_1,\cdot)$ are holomorphic in $\mathbb{D}\smallsetminus\{\zeta_1\}$ for each $\zeta_1\in\mathbb{D}$;
\item the functions $\ppsi{\ssp\sspm}(\cdot,\zeta_2)$ are holomorphic in $\mathbb{D}\smallsetminus\{\zeta_2\}$ for each $\zeta_2\in\mathbb{D}$.
\end{itemize}

\begin{proposition}[{cf.~\cite[Proposition~7.1]{CLR1}}] \label{prop:fpmpm} In the setup of Theorem~\ref{thm:main-GFF}, the following asymptotics as~$\zeta_2\to\zeta_1\in\mathbb{D}$ hold for all subsequential limits~\eqref{eq:fpmpm-def}:
\begin{equation}
\label{eq:fpmpm-res}
\ppsi{\ssp\ssp}(\zeta_1,\zeta_2)=\frac{2}{\pi i}\cdot \frac{1}{\zeta_2-\zeta_1}+O(1)\quad \text{and}\quad \ppsi{\ssm\ssp}(\zeta_1,\zeta_2)=O(1).
\end{equation}
\end{proposition}
\begin{proof}
Fix a point~$\zeta_1\in\mathbb{D}$ and consider a sequence~$w^\delta\to z_1:=z(\zeta_1)$. Assume also that the directions~$\eta_{w^\delta}$ converge to a certain direction~$\eta\in\mathbb{T}$ as~$\delta\to 0$ (due to compactness arguments this can be always done by passing to a subsequence). Due to Proposition~\ref{prop:Fpmpm-def}(ii) this implies the convergence
\[
F_{\!w^\delta}^\tw(\cdot)\ \to\ \fw(\cdot):=\tfrac12\big(\overline{\eta}\ff{\ssp\ssp}(z_1,\cdot)+\eta\ff{\ssm\ssp}(z_1,\cdot)\big)\ \ \text{as}\ \ \delta\to 0,
\]
uniformly on compact subsets of~$\Omega_\xi\smallsetminus\{z_1\}$. Let
\[
\psi_\frw(\zeta)\ :=\ \aa\ssp\ssp(\zeta)\cdot \fw(z(\zeta))+\aa\ssm\ssp(\zeta)\cdot \overline{\fw(z(\zeta))}\,;
\]
recall that the function~$\psi_\frw$ is holomorphic in~$\mathbb{D}\smallsetminus\{\zeta_1\}$ due to Proposition~\ref{prop:hol-in-zeta}.
A straightforward computation shows that
\[
\psi_\frw(\cdot)\ =\ \tfrac12\big(\overline{\mu}\ppsi{\ssp\ssp}(\zeta_1,\cdot)+\mu\ppsi{\ssm\ssp}(\zeta_1,\cdot)\big)
\]
provided that~$\overline{\eta}=\overline{\mu}\bb\ssp\ssp(\zeta_1)+\mu\bb\ssp\ssm(\zeta_1)$, which can be written as
\[
\mu\ =\ \frac{\eta\cdot\bb\ssp\ssp(\zeta_1)-\overline{\eta}\cdot\bb\ssm\ssp(\zeta_1)}{|\bb\ssp\ssp(\zeta_1)|^2-|\bb\ssm\ssp(\zeta_1)|^2}\,.
\]
The convergence of t-white-holomorphic function~$F_{\!w^\delta}$ to~$\fw$ implies that
\begin{itemize}
\item $\int (\fw(z)dz+\overline{\fw(z)}d\overline{\olim(z)})$ has an additive monodromy~$2\overline{\eta}$ around~$z_1$;
\item moreover, the projection of this primitive onto the direction~$i\overline{\eta}$ satisfies a one-sided maximum principle in a vicinity of~$z_1$.
\end{itemize}

Let us now consider a constant function~$\fb:=i\eta$ and denote, as usual,
\[
\psi_\frb(\zeta)\ := \bb\ssp\ssp(\zeta)\cdot \fb(z(\zeta))+\bb\ssm\ssp(\zeta)\cdot \overline{\fb(z(\zeta))}\ =\ i(\eta\bb\ssp\ssp(\zeta)-\overline{\eta}\bb\ssm\ssp(\zeta));
\]
recall that~$\psi_\frb$ is holomorphic in~$\mathbb{D}$ and note that
\[
\psi_\frb(\zeta_1)\ =\ i\mu\cdot (|\bb\ssp\ssp(\zeta_1)|^2-|\bb\ssm\ssp(\zeta_1)|^2).
\]
We know that the harmonic function
\begin{align*}
\textstyle \Re\big[i\eta\int (\fw dz+\ofw d\overline{\olim})\big]\ =\ \int\Re[\,\fb\fw dz+\ofb\fw d\olim\,]\ =\ \int\Re[\,\psi_\frb\psi_\frw d\zeta\,]
\end{align*}
is well defined and satisfies a one-sided maximum principle near~$\zeta_1$. Therefore, the function~$\psi_\frw\psi_\frb$ has at most a simple pole at~$\zeta_1$, which means that the function
\begin{align*}
\psi_\frb(\zeta_1)\psi_\frw(\cdot)\ &=\ \tfrac{1}{2}\big(\overline{\mu}\psi_\frb(\zeta_1)\cdot \ppsi{\ssp\ssp}(\zeta_1,\cdot)+\mu\psi_\frb(\zeta_1)\cdot \ppsi{\ssm\ssp}(\zeta_1,\cdot)\big)\\
&=\ \tfrac{i}{2}(|\bb\ssp\ssp(\zeta_1)|^2-|\bb\ssm\ssp(\zeta_1)|^2)\cdot (|\mu|^2\ppsi{\ssp\ssp}(\zeta_1,\cdot)+\mu^2\ppsi{\ssm\ssp}(\zeta_1,\cdot))
\end{align*}
also has at most a simple pole at~$\zeta_1$. Moreover, its residue at~$z_1$ has to be purely real. It is easy to see from~\cite[Lemma~6.3]{CLR1} that each small (of size~$O(\delta)$) vicinity of~$z_1$ contains two points~$w^\delta_1,w^\delta_2$ such that~$2-\varepsilon\ge |\eta_{w^\delta_1}-\eta_{w^\delta_2}|\ge \varepsilon$, where~$\varepsilon>0$ can depend on~$z_1$ but is independent of~$\delta$. By varying~$\eta$, this observation allows us to conclude that the first term~$\ppsi{\ssp\ssp}(\zeta_1,\cdot)$ has at most a simple pole at~$\zeta_1$, with a purely imaginary residue, while the second term~$\ppsi{\ssm\ssp}(\zeta_1,\cdot)$ does not have any singularity.

It remains to compute this residue. Consider another constant function $\fb:=\eta$ instead of~$\fb=i\eta$ and recall that the additive monodromy of the primitive $\int \Re[\eta(\fw dz+\ofw d\overline\olim)]$ around~$z_1$ equals~$2$. A straightforward computation similar to the one made above shows that this monodromy is
\begin{align*}
\tfrac 12\Re\big[\,&\overline\mu\cdot (\eta\bb\ssp\ssp(\zeta_1)+\overline{\eta}\bb\ssm\ssp(\zeta_1))\cdot 2\pi i\,\mathrm{res}_{\zeta_1}\ppsi{\ssp\ssp}(\zeta_1,\cdot)\,\big]\\
& =\ \Re\big[\,\pi i\,\mathrm{res}_{\zeta_1}\ppsi{\ssp\ssp}(\zeta_1,\cdot)\,\big]\ =\ \pi i\,\mathrm{res}_{\zeta_1}\ppsi{\ssp\ssp}(\zeta_1,\cdot)
\end{align*}
since we already know from the preceding reasoning that the residue of the function $\ppsi{\ssp\ssp}(\zeta_1,\cdot)$ belongs to~$i\R$. The proof of~\eqref{eq:fpmpm-res} is complete.
\end{proof}
Following~\cite[Section~7]{CLR1}, let us introduce differential forms
\begin{equation}
\label{eq:An-def}
\mathcal{A}_n(\zeta_1,\ldots,\zeta_n)\,:=\,4^{-n}\hskip -12pt\sum_{{s_1,\ldots,s_n\in\{\pm\}}}\det\big[\mathbbm{1}_{j\ne k}\ppsi{s_j,s_k}(\zeta_j, \zeta_k)\big]_{j,k=1}^n{\prod_{k=1}^n} d\zzeta{s_k}_k\,,
\end{equation}
where~$d\zzeta\ssp_k:=d\zeta_k$ and~$d\zzeta\ssm_k:=d\overline{\zeta}_k$. The next proposition repeats the classical computation of Kenyon~\cite{kenyon-gff-a,kenyon-gff-b} on t-embeddings that converge to a Lorentz-minimal surface in~$\R^{2+2}$.
\begin{proposition}[{cf.~\cite[Proposition~7.2]{CLR1}}]
\label{prop:dH-limit}
In the setup of Theorem~\ref{thm:main-GFF}, assume that the functions~$\FF{\pm\pm}$ have subsequential limits~\eqref{eq:fpmpm-def}. Let~$\ppsi{\sspm\sspm}$ be defined by~\eqref{eq:psipmpm-def} and the differential form~$\mathcal{A}_n$ be given by~\eqref{eq:An-def}. Then,
\begin{align}
\sum_{r_1,\ldots,r_n\in\{1,2\}}&(-1)^{r_1+\ldots r_n}\,H_n^\delta(v^\delta_{1,r_1},\ldots,v^\delta_{n,r_n}) \notag\\[-4pt]  &\to\ \int_{\zeta(v_{1,1})}^{\zeta(v_{1,2})}\!\!\ldots\int_{\zeta(v_{n,1})}^{\zeta(v_{n,2})}\mathcal{A}_n(\zeta_1,\ldots,\zeta_n)
\label{eq:dH-limit}
\end{align}
as~$\delta\to 0$. The multiple integral can be evaluated over an arbitrary collection of pairwise non-intersecting paths~$\zeta(\gamma_k)$ linking~$\zeta(v_{k,1})$ and~$\zeta(v_{k,2})$; in particular, $\mathcal{A}_n$ is an exact form in each of the variables~$\zeta_1,\ldots, \zeta_n\in\mathbb{D}$. {The~convergence is uniform provided that the paths~$\gamma_1,\ldots,\gamma_n\subset\Omega_\xi$ remain at a definite distance from each other and from~$\partial\Omega_\xi$.}
\end{proposition}

\newcommand\cTT[2]{\cT_{{\scriptscriptstyle [}#1{\scriptscriptstyle]}}^{{\scriptscriptstyle [}#2{\scriptscriptstyle ]}}}

\begin{proof} The main part of the proof mimics~\cite[Proposition 7.2]{CLR1}; we repeat it here for completeness of the exposition. Let $\gamma_k^\delta$ be a path running over edges of the t-embedding~$\cT^\delta$ near~$\gamma_k$, from~$v_{k,1}$ to~$v_{k,2}$. Note that, in general, we do \emph{not} control the total length of these paths as we do not assume that the angles of t-embeddings are uniformly bounded from below. Let~$(b^\delta_k w^\delta_k)^*\in\gamma^\delta_k$ be some edges on these paths. Identity~\eqref{eq:P-k-dimers} implies that
\begin{align*}
{\sum_{r_1,\ldots,r_n\in\{1,2\}}}&{(-1)^{r_1+\ldots+r_n} H^\delta_n(v^\delta_{1,r_1},\ldots,v^\delta_{n,r_n})} \\[-4pt]
& =\ \int_{v^\delta_{1,1}}^{v^\delta_{1,2}}\!\!\ldots\int_{v^\delta_{n,1}}^{v^\delta_{n,2}} \det\big[\mathbbm{1}_{j\ne k}K^{-1}_{ \cT^\delta}(w_j^\delta,b_k^\delta)\big]_{j,k=1}^n\cdot {\prod_{k=1}^n} d\cT^\delta_k,
\end{align*}
where the integrals are computed along the paths~$\gamma_1^\delta,\ldots,\gamma_n^\delta$ and the increment~$d\cT^\delta_k=\pm d\cT^{\delta}((b^\delta_k w^\delta_k)^*)$ is always oriented from~$v^\delta_{k,1}$ to~$v^\delta_{k,2}$ (the~$\pm$ sign depends on whether~$b^\delta_k$ is to the right or to the left from~$\gamma_k$). The diagonal $j=k$ is excluded since we are interested in the correlations of the fluctuations~$\hbar^\delta(v_k^\delta)=h^\delta(v_k^\delta)-\mathbb E[h^\delta(v_k^\delta)]$ and not in the functions~$h^\delta(v_k^\delta)$ themselves.

Expanding the determinant one obtains {the following expression:}
\begin{align*}
\sum_{\sigma\in S_n: \sigma(k)\ne k}(-1)^{\operatorname{sign}(\sigma)} \int_{v^\delta_{1,1}}^{v^\delta_{1,2}}\!\!\ldots\int_{v^\delta_{n,1}}^{v^\delta_{n,2}} \ \prod_{k=1}^nK^{-1}_{\cT^\delta}(w_k^\delta,b_{\sigma(k)}^\delta)d\cT^\delta_k,
\end{align*}
where the sum is taken over all permutations~$\sigma$ with no fixed points.
Denote
\[
d\cTT\ssp\ssp:=d\cT, \quad d\cTT\ssp\ssm:=d\overline{\cT}, \quad  d\cTT\ssm\ssp:=d\cO\ \ \text{and}\ \ d\cTT\ssm\ssm:=d\overline{\cO}.
\]
As in the proof of~\cite[Proposition 7.2]{CLR1}, we now pass from real-valued dimer coupling functions~$K^{-1}_{\cT^\delta}$ to the complexified ones~$\FF{\pm\pm}_{\cT^\delta}$; see Proposition~\ref{prop:Fpmpm-def} and~\eqref{eq:dO-on-T}. This gives the identity
\[
\prod_{k=1}^n K^{-1}_{\cT^\delta}(w^\delta_k,b^\delta_{\sigma(k)})d\cT^\delta_k\ =\ 4^{-n}\!\!\!\sum_{p_k,q_k\in\{\pm\}}\ \prod_{k=1}^n \FF{p_k,q_{\sigma(k)}}(b^\delta_k,w^\delta_{\sigma(k)})d\cT^{[q_k],\delta}_{[p_kq_k],k}\,,
\]
where the sum is taken over all~$2^{2n}$ possible combinations of signs~$p_k$ and~$q_k$. This closed differential form can be extended from edges of the t-embedding into the plane, in all the variables~$z_k$ simultaneously. In particular, the paths~$\gamma_k^\delta$ can now be assumed to coincide with~$\gamma_k$ except near the endpoints and, in particular, to have uniformly (in~$\delta$) bounded lengths.

The convergence~\eqref{eq:fpmpm-def} allows us to conclude that
\begin{align*}
 {\sum_{r_1,\ldots,r_n\in\{1,2\}}}\,&{(-1)^{r_1+\ldots+r_n} H^\delta_n(v^\delta_{1,r_1},\ldots,v^\delta_{n,r_n})} \\[-4pt]
 &\to\ \ 4^{-n}\!\!\!\!\!\!\sum_{\sigma\in S_n: \sigma(k)\ne k} (-1)^{\operatorname{sign}(\sigma)}\int_{v_{1,1}}^{v_{1,2}}\!\!\ldots\int_{v_{n,1}}^{v_{n,2}}\mathcal{A}_n^{(\sigma)}(z_1,\ldots,z_n)
\end{align*}
as~$\delta\to 0$, where
\[
\mathcal{A}_n^{(\sigma)}(z_1,\ldots,z_n)\ := \sum_{p_k,q_k\in\{\pm\}}\ \prod_{k=1}^n\ff{p_k,q_{\sigma(k)}}(z_k,z_{\sigma(k)})dz^{[q_k]}_{[p_kq_k],k}
\]
and we use the notation~$d\zz\ssp\ssp:=dz$, $d\zz\ssp\ssm:=d\overline{z}$, $d\zz\ssm\ssp:=d\olim$, \mbox{$d\zz\ssm\ssm:=d\overline{\olim}$}.
To complete the proof, it remains to write this differential form using the conformal coordinates $\zeta_k\in\mathbb{D}$. Recall (see~\eqref{eq:dzz=bbaadzzeta}) that
\[
 d\zz{pq}{q}\ =\ \bb{p}{\ssp}\aa{q}{\ssp}d\zeta\ +\ \bb{p}{\ssm}\aa{q}{\ssm}d\overline{\zeta}\ =\ \sum_{s\in\{\pm\}}\bb{p}{s}\aa{q}{s}d\zzeta{s}\,.
\]
Therefore, for each~$\sigma\in S_n$ with no fixed points we have
\begin{align*}
\mathcal{A}_n^{(\sigma)}\ &=\sum_{p_k,q_k,s_k\in\{\pm\}}\ \prod_{k=1}^n \bb{p_k}{s_k}(\zeta_k)\aa{q_k}{s_k}(\zeta_k)\ff{p_k,q_{\sigma(k)}}(z_k,z_{\sigma(k)})d\zzeta{s_k}_k\\
&=\sum_{s_k\in\{\pm\}}\ \prod_{k=1}^n\sum_{p_k,q_k\in\{\pm\}} \bb{p_k}{s_k}(\zeta_k)\aa{q_{\sigma(k)}}{s_{\sigma(k)}}(\zeta_{\sigma(k)})\ff{p_k,q_{\sigma(k)}}(z_k,z_{\sigma(k)})d\zzeta{s_k}_k\\
&=\sum_{s_k\in\{\pm\}}\ \prod_{k=1}^n \ppsi{s_k,s_{\sigma(k)}}(\zeta_k,\zeta_{\sigma(k)})d\zzeta{s_k}_k\,,
\end{align*}
which completes the proof since~$4^{-n}\sum_{\sigma\in S_n:\sigma(k)\ne k}(-1)^{\operatorname{sign}(\sigma)}\mathcal{A}_n^{(\sigma)}$ is nothing but the definition~\eqref{eq:An-def} of the differential form~$\mathcal{A}_n$.
\end{proof}
We are now ready to prove the main result of our paper.
\begin{proof}[Proof of Theorem~\ref{thm:main-GFF}] Due to Proposition~\ref{prop:dH-limit} we only need to show that
\begin{equation}
\label{eq:An=dGn}
\mathcal{A}_n(\zeta_1,\ldots,\zeta_n)\ =\ \pi^{-n/2}d_{\zeta_1}\ldots d_{\zeta_n}G_{\mathbb{D},n}(\zeta_1,\ldots,\zeta_n)
\end{equation}
for all subsequential limits~\eqref{eq:fpmpm-def} of the functions~$\FF{\pm\pm}_{\cT^\delta}$. The key observation that allows us to do so is that the coefficients of the (piece-wise constant) differential forms
\begin{equation}\label{eq:x-An-discrete}
4^{-n}\!\!\!\sum_{p_k,q_k\in\{\pm\}}\ \prod_{k=1}^n \FF{p_k,q_{\sigma(k)}}(b_k,w_{\sigma(k)})d\cT^{[q_k],\delta}_{[p_kq_k],k}
\end{equation}
are uniformly bounded not only in the situation when all edges~$(b^\delta_kw^\delta_k)$ stay away from each other and from $\partial\Omega_\xi$ as~$\delta\to 0$ but also when \emph{one} of them (say, that with~$k=n$) is allowed to approach~$\partial\cT^\delta$; see Theorem~\ref{thm:K-1=O(1)}.

Let~$v_{n,0}^\delta$ be the closest point to~$v_{n,1}^\delta$ at the polygonal boundary of the perfect t-embedding~$\partial\cT^\delta$. The standard computation given in the proof of Proposition~\ref{prop:dH-limit} with~$v^\delta_{n,2}$ replaced by~$v^\delta_{n,0}$ implies that
\begin{align}
{\sum_{r_1,\ldots,r_{n-1}\in\{1,2\}}}\!\!{(-1)^{r_1+\ldots+r_{n-1}} H^\delta_n(v^\delta_{1,r_1},\ldots,v^\delta_{n-1,r_{n-1}},v^\delta_{n,1})}& \notag \\[-4pt]
=\ O(\mathrm{dist}(v^\delta_{n,1},v^\delta_{n,0}))&
\label{eq:x-estim}
\end{align}
uniformly \emph{both} in~$\delta$ and in~$v^\delta_{n,1}$ provided that all other points stay away from~$\partial\Omega_\xi$ and from each other. (Note that there is no need to assume that~$v_{n,0}^\delta$ is a boundary \emph{vertex} of~$\cT^\delta$ since the primitive of the piece-wise constant differential form~\eqref{eq:x-An-discrete} vanishes everywhere at the boundary of~$\cT^\delta$.) Combining this uniform estimate and Proposition~\ref{prop:dH-limit} we see that in fact the following (stronger than~\eqref{eq:dH-limit}) convergence result holds:
\begin{align}
{\sum_{r_1,\ldots,r_{n-1}\in\{1,2\}}}\!\!&{(-1)^{r_1+\ldots+r_{n-1}} H^\delta_n(v^\delta_{1,r_1},\ldots,v^\delta_{n-1,r_{n-1}},v^\delta_{n,2})} \notag \\[-4pt]
&\mathop{\to}\limits_{\delta\to 0}\quad \lim_{v_{n,1}\to\partial\Omega_\xi} \int_{\zeta(v_{1,1})}^{\zeta(v_{1,2})}\ldots\int_{\zeta(v_{n,1})}^{\zeta(v_{n,2})}\mathcal{A}_n(\zeta_1,\ldots,\zeta_n)\,;
\label{eq:x-conv}
\end{align}
in particular, this limit as~$v_{n,1}\to\partial\Omega$ exists and does not depend on the way in which~$v_{n,1}$ approaches the boundary of~$\Omega_\xi$ or, equivalently, on the way in which~$\zeta(v_{n,1})$ approaches the boundary of the unit disc~$\mathbb{D}$.

Let us now assume that~$n=2$, fix two distinct points~$\zeta_{1,1},\zeta_{1,2}\in\mathbb{D}$ and consider the function
\[
h_2(\zeta_{1,1},\zeta_{1,2}\,;\,\zeta)\ :=\ \lim_{\widetilde{\zeta}\to\partial\mathbb{D}}
\int_{\zeta_{1,1}}^{\zeta_{1,2}}\int_{\widetilde{\zeta}}^{\zeta}\mathcal{A}_2(\zeta_1,\zeta_2).
\]
This function is harmonic for~$\zeta\in\mathbb{D}\smallsetminus \{\zeta_{1,1},\zeta_{1,2}\}$ and has Dirichlet boundary values as~$\zeta\to \partial\mathbb{D}$ due to the same combination of the uniform estimate~\eqref{eq:x-estim} and the convergence~\eqref{eq:x-conv}. Moreover, it follows from Proposition~\ref{prop:fpmpm} that
\[
\mathcal{A}_2(\zeta_1,\zeta_2)\ =\ -\frac{1}{2\pi^2}\Re\biggl(\frac{d\zeta_1 d\zeta_2}{(\zeta_2-\zeta_1)^2}\biggr)+O\biggl(\frac{|d\zeta_1||d\zeta_2|}{|\zeta_2-\zeta_1|}\biggr)\ \ \text{as}\ \ \zeta_2\to\zeta_1\,.
\]
This asymptotics allows us to identify the behavior of $h_2(\zeta_{1,1},\zeta_{1,2};\,\cdot\,)$ at singularities and to prove that
\[
h_2(\zeta_{1,1},\zeta_{1,2};\,\cdot\,)\ =\ G_\mathbb{D}(\zeta_{1,2},\,\cdot\,)-G_\mathbb{D}(\zeta_{1,1},\,\cdot\,),
\]
where~$G_\mathbb{D}$ is the Green function in~$\mathbb{D}$ with Dirichlet boundary conditions. Therefore,
\begin{equation}
\label{eq:A2=dG2}
\mathcal{A}_2(\zeta_1,\zeta_2)\ =\ \pi^{-1}d_{\zeta_1}d_{\zeta_2}G_{\mathbb{D}}(\zeta_1,\zeta_2)\,.
\end{equation}

A similar argument for~$n=3$ yields the identity
\begin{equation}
\label{eq:A3=dG3}
\mathcal{A}_3(\zeta_1,\zeta_2,\zeta_3)\ =\ 0
\end{equation}
since in this case Proposition~\ref{prop:fpmpm} implies that the differential form~$\mathcal{A}_3$ has at most linear singularities. (This means that its primitive~\eqref{eq:x-conv}, viewed as a harmonic function of the last variable, have \emph{no} singularities and thus vanishes due to the Dirichlet boundary conditions.) It remains to note that~\eqref{eq:A2=dG2} and~\eqref{eq:A3=dG3}, together with the definition~\eqref{eq:An-def} of the differential forms~$\mathcal{A}_n$, imply that the identity~\eqref{eq:An=dGn} also holds for all~$n\ge 4$; see~\cite[Lemma~7.3]{CLR1}. 
\end{proof}

\section{Discussion and open questions}\label{sec:discussion}

\subsection{Finding a perfect Coulomb gauge for a given weighted graph} \label{sub:finding-gauge}
Let~$\G$ be a given (big) weighted bipartite graph with a marked vertex~$v_\mathrm{out}$ of its dual graph, which is then augmented at~$v_\mathrm{out}$ and denoted by~$\G^*$ as in Section~\ref{sub:origami}. In what follows we denote the positive weights of its edges by~$\chi_e$ and by~$K_\R:\R^W\to\R^B$ a \emph{real-valued} Kasteleyn matrix on this graph, i.e., a matrix with entries~$K(b,w)=\pm\chi_{(bw)}$ such that the 
product of all {the `$\pm$'} signs around each vertex~$v$ of the dual graph~$\G^*$ equals~$(-1)^{{\frac12\mathrm{deg}v}-1}$. Following~\cite{KLRR}, we call~$\cF^\tb\!:\!B\!\to\!\C$ and~$\cF^\tw\!:\!W\!\to\!\C$ \emph{Coulomb gauge} functions if
\begin{equation}
\label{eq:Coulomb-def}
\begin{array}{ll}
[K^\top \cF^\tb](w)=0 &\text{for all~$w\in W\smallsetminus\partial W$,}\\[2pt]
[K\cF^\tw](b)=0 &\text{for all~$b\in B\smallsetminus\partial B$.}
\end{array}
\end{equation}
In this case, one defines a \emph{t-realisation}~$\cT=\cT_{(\cF^\tb,\,\cF^\tw)}$ of the graph~$\G^*$ in the complex plane and the associated origami map~$\cO=\cO_{(\cF^\tb,\,\cF^\tw)}$ by setting
\begin{equation}
\label{eq:TO-def-via-F}
\begin{array}{rcl}
d\cT(bw^*)&\!:=\!& \cF^\tb(b)K_\R(b,w)\cF^\tw(w),\\[2pt]
d\cO(bw^*)&\!:=\!&\cF^\tb(b)K_\R(b,w)\overline{\cF^\tw(w)}\,;
\end{array}
\end{equation}
note that~\eqref{eq:Coulomb-def} implies that these $1$-forms are closed around each face of~${\G^*}$ and hence define~$\cT$ and~$\cO$ up to global additive constants. The word `realisation' used instead of `embedding' emphasizes the fact that we do not insist on the fact that~$\cT$ is a \emph{proper} embedding of~$\G^*$. 

Further, we call a Coulomb gauge~$(\cF^\tb,\cF^\tw)$ \emph{non-degenerate} if~$\cF^\tb(b)\ne 0$ for all~$b\in B$, $[K^\top\cF^\tb](w)\ne 0$ for all~$w\in\partial W$, $F^\tw(w)\ne 0$ for all~$w\in W$, and~$[K\cF^\tw](b)\ne 0$ for all~$b\in \partial B$. This means that no edge of~$\G^*$, including those of the outer face, degenerates under~$\cT$.

{Recall that~$\rH$ defines the one-sheet hyperboloid~\eqref{eq:hyperboloid-def}, which is nothing but the unit sphere in the Minkowski space~$\R^{2,1}$ viewed as a subspace of~$\R^{2,2}$.} As in Section~\ref{sub:origami} we use the notation~$v_1,\ldots,v_{2n}$ for the outer vertices of~$\G^*$ labeled counterclockwise. {We also denote} by~$v_{{\mathrm{in},k}}$ the (unique) inner vertex of~$\G^*$ that is adjacent to~$v_k$.

\begin{theorem}\label{thm:perfect-gauge}
A non-degenerate Coulomb gauge~$(\cF^\tb,\cF^\tw)$, {considered up to simultaneous multiplications~$(\cF^\tb,\cF^\tw)\mapsto(\lambda\cF^\tb,\lambda^{-1}\cF^\tw)$, $\lambda\in\C$, defines a perfect t-embedding if and only if one can choose~$\lambda$ and the additive constants of integration in the definition~\eqref{eq:TO-def-via-F} so that} the following properties hold:
\begin{enumerate}\renewcommand\labelenumi{(\roman{enumi})}
\item the point $(\cT(v_k),\cO(v_k))$ lies on the hyperboloid~\eqref{eq:hyperboloid-def} for all~$k$\,;

\vskip 2pt

\item ${\pm}\Im\!\big[(\cT(v_{k{\pm} 1})-\cT(v_k))\cdot {\overline{d\cT((v_{\mathrm{in},k}v_k))}}\,\big]\!>\!0$ {for all~$k$ and both signs~$\pm$\,;}

\vskip 2pt

\item the index of the polyline~$\cT(v_1)\cT(v_2)\ldots\cT(v_{2n})\cT(v_1)$ with respect to the origin is~$1$ (i.e., the boundary~$\partial\cT$ winds around~$0$ exactly once).
\end{enumerate}
\end{theorem}
Clearly, all the properties listed above hold for perfect t-embeddings, so the main content of Theorem~\ref{thm:perfect-gauge} is the converse statement: if we know that~$(\cT(v_k),\cO(v_k))\in\rH$ for all~$k=1,\ldots,2n$, then~$\cT$ is automatically a \emph{proper} embedding provided that the mild assumptions {(ii)} and (iii) holds.

\begin{remark}\label{rem:Lorentz-isometries}
The hyperboloid~$\rH$ is invariant under rotations~$z\mapsto\alpha^2 z$, $\alpha\in\mathbb{T}$, which correspond to multiplications~$\cF^\tb\mapsto \alpha\cF^\tb$, $\cF^\tw\mapsto\alpha\cF^\tw$ and is also invariant under Lorentz boosts
\[
x\mapsto x\cosh(\varsigma) -\vartheta\sinh(\varsigma)\,,\quad y\mapsto y\,,\quad \vartheta\mapsto \vartheta\cosh(\varsigma)-x\sinh(\varsigma),
\]
where~$z=x+iy$ and~$\varsigma\in\R$. The latter correspond to a simultaneous change
\[
\cF^\tb\mapsto \cF^\tb\cosh(\tfrac12\varsigma) -\overline{\cF}{}^\tb\sinh(\tfrac12\varsigma)\,,\quad \cF^\tw\mapsto \cF^\tw\cosh(\tfrac12\varsigma) -\overline{\cF}{}^\tw\sinh(\tfrac12\varsigma)
\]
of the gauge functions.
\end{remark}
The proof of Theorem~\ref{thm:perfect-gauge} follows the lines of~\cite[Theorem~4.6]{kenyon-sheffield} and~\cite[Lemma~14]{KLRR}, with a certain adaptation to the context of this paper. We begin with two preliminary lemmas.
\begin{lemma}\label{lem:boundary-on-H} In the setup of Theorem~\ref{thm:perfect-gauge} there exists angles
\[
\phi_1<\ldots<\phi_{{2n}}<\phi_{{2n+1}}=\phi_1+2\pi\ \ \text{and}\ \ \xi_k\in (-\tfrac\pi 2,\tfrac\pi 2),
\]
$k=1,\ldots,2n$, satisfying the identities~\eqref{eq:phi-phi=xi-xi} with~$\xi_{{2n+1}}:=\xi_1$ such that
\begin{equation}
\label{eq:TO-via-phixi}
\begin{array}{ll}
\cT(v_k)=e^{i\phi_k}/\cos\xi_k\,, & {\pm}\cO(v_k)=\tan\xi_k\,,\\[4pt]
d\cT((v_{{\mathrm{in},k}}v_k))\in e^{i\phi_k}\R_+\,,\ \ &  {\pm}d\cO((v_{{\mathrm{in},k}}v_k))\in -ie^{i\xi_k}\R_+
\end{array}
\end{equation}
{for all outer vertices~$v_k$ of~$\G^*$,}{where the `$\pm$' sign does not depend on~$k$ and can be fixed to be~`$+$' by applying the multiplication $(\cF^\tb,\cF^\tw)\mapsto(i\cF^\tb,-i\cF^\tw)$.}
\end{lemma}
\begin{proof} By construction, for each~$k=1,\ldots,2n$ the points~$(\cT(v_k),\cO(v_k))$ and~$(\cT(v_{k+1}),\cO(v_{k+1}))$ are in a light-like position in the Minkowski space~$\R^{2,1}$. Since they both lie on the hyperboloid~\eqref{eq:hyperboloid-def}, the segment joining them should belong to one of the lines
\begin{equation}
\label{eq:lines-on-H}
\begin{array}{l}
\big(e^{i(\phi+t)}/\cos(\xi\!+\!t)\,,\,\tan(\xi\!+\!t)\big)_{t\in\R}\,,\\[2pt]
\big(e^{i(\phi+t)}/\cos(\xi\!-\!t)\,,\,\tan(\xi\!-\!t)\big)_{t\in\R}
\end{array}
\end{equation}
that form two pencils generating~$\rH$. Note that the directions of these lines (in~$\R^{2,1}\cong\C\times\R$) are $(ie^{i(\phi-\xi)},1)$ and $(ie^{i(\phi+\xi)},-1)$, respectively.

Recall that
\[
\cO(v_{2k})-\cO(v_{2k-1})\ =\ \frac{(\overline{\cF^\tw(w_k)})^2}{\lvert \cF^\tw(w_k) \rvert^2}\cdot (\cT(v_{2k})-\cT(v_{2k-1}))
\]
for all~$k=1,\ldots,n$ due to~\eqref{eq:TO-def-via-F} and, similarly, that
\begin{align*}
\cO(v_{2k+1})-\cO(v_{2k})\ &=\ \frac{(\cF^\tb({b_k}))^2}{\lvert \cF^\tb({b_k}) \rvert^2}\cdot (\overline{\cT(v_{2k+1})}-\overline{\cT(v_{2k})})\\[2pt]
&=\ \frac{(\overline{\cF^\tb({b_k})})^2}{\lvert \cF^\tb({b_k}) \rvert^2}\cdot (\cT(v_{2k+1})-\cT(v_{2k}))
\end{align*}
since all increments of~$\cO$ along the outer face of~$\G^*$ are purely real. It is now easy to see that the images of boundary edges~$(v_{2k}v_{2k\pm 1})$ under the mapping~$(\cT,\cO)$ should belong to \emph{different} pencils of lines~\eqref{eq:lines-on-H}. Indeed, if they both belonged to, say, a first line in~\eqref{eq:lines-on-H}, then we would have both
\[
(\cF^\tw(w_k))^2\ \in\ ie^{i(\phi-\xi)}\R_+\quad \text{and}\quad (\cF^\tb({b_k}))^2\ \in\ ie^{i(\phi-\xi)}\R_+,
\]
and hence~$\cT(v_{2k})-\cT(v_{{\mathrm{in},2k}})= -\cF^\tb({b_k})K_\R({b_k},w_k)\cF^\tw(w_k)\in ie^{i(\phi-\xi)}\R$. Thus, this increment would be aligned with~$\cT(v_{2k+1})-\cT(v_{2k})$, which contradicts to the condition {(ii)} in Theorem~\ref{thm:perfect-gauge}.
A similar argument applies to the images of boundary edges~$(v_{2k}v_{2k+1})$ and~$(v_{2k+1}v_{{2k+2}})$.

{Assume that the images of edges~$(v_{2k-1}v_{2k})$ belong to lines from the first pencil in~\eqref{eq:lines-on-H} while the images of edges~$(v_{2k}v_{2k+1})$ belong to the second one; the other case can be treated by replacing~$\cO$ with~$-\cO$.} {Let~$(ie^{i\beta^\tw_k},1)$ and~$(ie^{i\beta^\tb_k},-1)$ be the directions of these lines. Then, their intersection points can be written as
\[
\cT(v_k)\,=\,e^{i\phi_k}/\cos\xi_k\,,\quad \cO(v_k)\,=\,\tan\xi_k\,,\quad k=1,\ldots,2n,
\]
where~$\phi_{2k}-\xi_{2k}=\phi_{2k-1}-\xi_{2k-1}=\beta^\tw_k$ and~$\phi_{2k}+\xi_{2k}=\phi_{2k+1}+\xi_{2k+1}=\beta^\tb_k$. In particular, the angles~$\phi_k$ and~$\xi_k$ satisfy the identities~\eqref{eq:phi-phi=xi-xi}. However, as for now $\phi_k$ and~$\xi_k$ are defined (modulo~$2\pi$) only up to a simultaneous addition of~$\pi$ to both of them. Now let us note that}
\[
{(\cF^\tw(w_k))^2\ \in\ ie^{i\beta_k^\tw}\R_+\quad \text{and}\quad (\cF^\tb(b_k))^2\ \in\ -ie^{i\beta_k^\tb}\R_+,}
\]
which {yields}~$d\cT((v_{{\mathrm{in},k}}v_k))\in e^{i\phi_k}\R$. {This} allows us to fix the {joint} ambiguity in the definition of~$\phi_k$ and~$\xi_k$ by requiring that~$d\cT((v_{{\mathrm{in},k}}v_k))\in e^{i\phi_k}\R_+$. {Note that this also gives, e.g.,~$d\cO((v_{\mathrm{in},2k}v_{2k}))\in -ie^{-i\beta_{2k}^\tw}e^{i\phi_{2k}}\R_+=-ie^{i\xi_{2k}}\R_+$ and similarly for~$d\cO((v_{\mathrm{in},2k-1}v_{2k-1}))$.} {It is also easy to see that the condition (ii) 
implies that} $\phi_{k+1}-\phi_k\in (0,\pi)$ and~$\cos\xi_{k+1}>0$ for all~$k=1,\ldots,2n$.

This completes the proof since the last claim~$\sum_{k=1}^{2n}(\phi_{k+1}-\phi_k)=2\pi$ is nothing but a reformulation of the {condition~(iii) in Theorem~\ref{thm:perfect-gauge}.}
\end{proof}

\begin{corollary} \label{cor:dTO-cyclic}
In the same setup, let~$\alpha\in\mathbb{T}$ be such that~$\alpha^2\ne -ie^{i(\phi_k-\xi_k)}$ for all~$k=1,\ldots,2n$. Then, the increments~$d(\cT+\alpha^2\cO)((v_{{\mathrm{in},k}}v_k))$ do not vanish and have cyclically ordered directions. A similar statement holds for the increments $d(\cT+\alpha^2\overline{\cO})((v_{{\mathrm{in},k}}v_k))$ provided that~$\alpha^2\ne ie^{i(\phi_k+\xi_k)}$ for all~$k$.
\end{corollary}
\begin{proof} A simple computation based upon~\eqref{eq:TO-via-phixi} shows that the increment $d(\cT\!+\!\alpha^2\cO)((v_{{\mathrm{in},k}}v_k))$ is proportional (with a positive real coefficient) to
\[
e^{i\phi_k}+\alpha^2 \cdot (-ie^{i\xi_k})\ =\ e^{-i\frac\pi 4}\alpha\cdot 2\Re\big[\,e^{i\frac\pi 4}\overline{\alpha}\,e^{\frac{i}2(\phi_k-\xi_k)} \big]\cdot e^{\frac{i}{2}(\phi_k+\xi_k)}.
\]
The fact that these directions are cyclically ordered follows from the fact that the directions~${ie^{i\beta^\tw_k}=ie^{i(\phi_{2k-1}-\xi_{2k-1})}=ie^{i(\phi_{2k}-\xi_{2k}})}$ of the lines containing the boundary edges~${(\cT(v_{2k-1})\cT(v_{2k}))}$ of~$\G^*$ are cyclically ordered and that the same holds for the directions~${-ie^{i\beta^\tb_k}=-ie^{i(\phi_{2k}+\xi_{2k})}=-ie^{i(\phi_{2k+1}+\xi_{2k+1})}}$. The consideration of the increments of the mapping~$\cT+\alpha^2\overline{\cO}$ is similar.
\end{proof}

\begin{lemma} \label{lem:T-no-strict-max} Let $F^\tb:B\to\R$ and~$F^\tw:W\to \R$ be real valued functions satisfying the identities~\eqref{eq:Coulomb-def} and let~$T=T_{(F^\tb,\,F^\tw)}:V(\G^*)\to\R$ be defined by~\eqref{eq:TO-def-via-F}. Then, the function~$T$ cannot have a strict local extremum at an inner vertex of~$\G^*$.
\end{lemma}
\begin{proof} This lemma is similar to~\cite[Proposition~3.6]{russkikh-t} and~\cite[Proposition~2.10]{chelkak-s-emb}. Let~$v$ be an inner vertex of~$\G^*$, $e_1,\ldots,e_{2m}$ be the edges emanating from this vertex, and~$b_1,\ldots,b_m\in B$, $w_1,\ldots,w_m\in W$ denote the faces of~$\G^*$ that are adjacent to~$u$. We have
\[
\textstyle \prod_{k=1}^{2m}dT(e_k)\ =\ -\prod_{k=1}^{2m}\chi_{e_k}\cdot\prod_{k=1}^m(F^\tb(b_k))^2\cdot \prod_{k=1}^m(F^\tw(w_k))^2\ \le\ 0\,,
\]
the minus sign comes from the fact that the entries~$K_\R(e_k)=\pm \chi_{e_k}$ of the matrix~$K_\R$ satisfy the Kasteleyn condition around~$v$.
\end{proof}
\begin{corollary} \label{cor:max-principle} Let~$(\cF^\tb,\cF^\tw)$ be non-degenerate Coulomb gauge functions and~$F^\tb(\cdot)=\Re[\overline\alpha \cF^\tb(\cdot)]$, $F^\tw(\cdot)=\Re[\overline{\beta}\cF^\tw(\cdot)]$, where~$\alpha,\beta\in\mathbb{T}$. Then, the function~$T=T_{(F^\tb\,,F^\tw)}$ satisfies the maximum principle: its extrema on any subgraph of~$\G^*$ are attained at the boundary of this subgraph.
\end{corollary}
\begin{proof} Assume first that {$F^\tb(b)\ne 0$ and $F^\tw(w)\ne 0$ for all~$b\in B$ and~$w\in W$.} In this case the function~$T_{(F^\tb,F^\tw)}$ does not have zero increments along edges of~$\G^*$ and thus satisfies the maximum principle due to Lemma~\ref{lem:T-no-strict-max}. In order to treat the general case one replaces~$F^\tb$ by~$F^\tb_\varepsilon:=\Re[e^{i\varepsilon} \overline{\alpha}\cF^\tb]$ and similarly for~$F^\tw$ with~$\varepsilon\to 0$. Since the complex-valued functions~$\cF^\tb$ and~$\cF^\tw$ have no zero, this allows one to destroy possible degeneracies, apply the maximum principle for the new function~$T\!=\!T_\varepsilon$, and then pass to the limit as~$\varepsilon\!\to\!0$.
\end{proof}

We are now in the position to prove Theorem~\ref{thm:perfect-gauge}.
\begin{proof}[Proof of Theorem~\ref{thm:perfect-gauge}] Recall that we only need to prove that the properties (i)--(iii) imply that~$\cT$ is a perfect t-embedding. Let~$P$ be the image of the outer face of~$\G^*$ under the mapping~$\cT$. It follows from Lemma~\ref{lem:boundary-on-H} that the (possibly, non-convex) polygon~$P$ is tangential to the unit circle and that the edges~$(v_{{\mathrm{in},k}}v_k)$ are mapped on the inner bisectors of its angles. The angle condition in Definition~\ref{def:t-emb} holds automatically provided that~$\cT$ is defined via~\eqref{eq:TO-def-via-F} and~$K_\R$ satisfies the Kasteleyn condition. Thus, we only need to show that~$\cT$ is a \emph{proper} embedding of~$\G^*$, i.e., that
\begin{enumerate}\renewcommand\theenumi{\alph{enumi}}
\item all faces of~$\G^*$ are mapped by~$\cT$ onto convex polygons;
\item all these polygons have the same orientation;
\item images of distinct faces of~$\G^*$ do not overlap and fill the interior of~$P$.
\end{enumerate}
This can be done following the lines of~\cite[Theorem~4.6]{kenyon-sheffield} and~\cite[Lemma~14]{KLRR} as we now explain.

To prove (a) consider, e.g., a face~$b\in B$ of~$\G^*$ and assume for a while that~$b\not\in\partial B$. Choose a generic value~$\alpha\in\mathbb{T}$ such that~$d(\cT+\alpha^2\overline{\cO})(e)\ne 0$ for all edges~$e$ of~$\G$. (Note that such a choice is always possible as we assume that the complex-valued gauge function~$\cF^\tw$ never vanishes.) The images of~$b$ under~$\cT$ and under~$\cT+\alpha^2\overline{\cO}$ are homothetical to each other, thus it suffices to prove that~$(\cT+\alpha^2\overline{\cO})(b)$ is a convex polygon. Consider now a function $T(v):=\Re[\overline{\beta}(\cT+\alpha^2\overline{\cO})(v)]$ defined on vertices of~$\G^*$, where~$\beta\in\mathbb{T}$ is a generic direction chosen so that~$dT(e)\ne 0$ for all edges of~$e$. Assume that \mbox{$(\cT+\alpha^2\overline{\cO})(b)$} is not a convex polygon. In this case one can choose~$\beta$ so that the values of $T$ at vertices of~$b$ are not cyclically ordered, i.e., there exist four cyclically order vertices~$y_0,y_1,y_2,y_3$ of~$b$ such that $T(y_p)>T(y_q)$ if~$p\in\{1,3\}$ and \mbox{$q\in\{0,2\}$}. However, it is easy to see that $T=T_{(F^\tb,F^\tw)}$, where~$F^\tb=\Re[\overline{\alpha}\cF^\tb]$ and~$F^\tw=\Re[\overline{\beta}\cF^\tw]$, and thus this function satisfies the maximum principle; see Corollary~\ref{cor:max-principle}. Hence, we can find four non-intersecting paths started at~$y_0,y_1,y_2,y_3$ and going to the boundary of~$\G^*$ along which the function~$T$ is decreasing, increasing, decreasing, increasing, respectively. This contradicts to Corollary~\ref{cor:dTO-cyclic} which implies that the increments~$dT((v_{{\mathrm{in},k}}v_k))$ are strictly positive along a certain boundary arc of $v_\mathrm{out}$ and strictly negative along the complementary arc.

It is not hard to check that the same proof also works if~$b\in\partial B$. Indeed, if, say,~$y_0=v_k$ is a boundary vertex of~$\G^*$, then either~$dT((v_{{\mathrm{in},k}}v_k))<0$, which does not break the argument, or \mbox{${T(v_{\mathrm{in},k})<T(v_k)}$} and hence one can take~$y_0:=v_{{\mathrm{in},k}}$ instead of~$v_k$. Therefore, $\cT(b)$ is a convex polygon for each~$b\in B$. The consideration of the images~$\cT(w)$, $w\in W$, is fully similar.

We now move to the proof of the property (b). It is enough to prove that, say, for each~$b\in B\smallsetminus\partial B$ all the faces~$\cT(w)$ with~$w\sim b$ have the same orientation. (Indeed, this implies that all faces~$\cT(w)$, $w\in W$, have the same orientation and similarly for~$\cT(b)$, $b\in B$; however, \emph{all} boundary faces share the same orientation.) Similarly to the proof of (a), choose a generic \mbox{$\alpha\in\mathbb{T}$} in order to avoid possible degeneracies and consider the image of the graph~$\G^*$ under the mapping~$T_\alpha:=\cT+\alpha^2\cO$. Recall that~$\cT_\alpha(b)$ is a segment obtained by projecting~$\cT(b)$ onto a certain direction depending on~$b$ while the images~$T_\alpha(w_k)$ are homothetical to~$\cT(w_k)$. Let~$x_k$ be vertices of~$b$ labeled in a cyclic order so that~$T_\alpha$ maps each of the boundary arcs $(x_1 x_2\ldots x_m)$ and $(x_m \ldots x_0 x_1)$ of~$b$ onto the segment $T_\alpha(b)=[T_\alpha(x_1),T_\alpha(x_m)]$; see Fig.~\ref{fig:orient}. Let~$w_k$ be the white face adjacent to~$b$ that shares the edge~$(x_k x_{k+1})$ with~$b$.

\begin{figure}
\center{\includegraphics[width=0.92\textwidth]{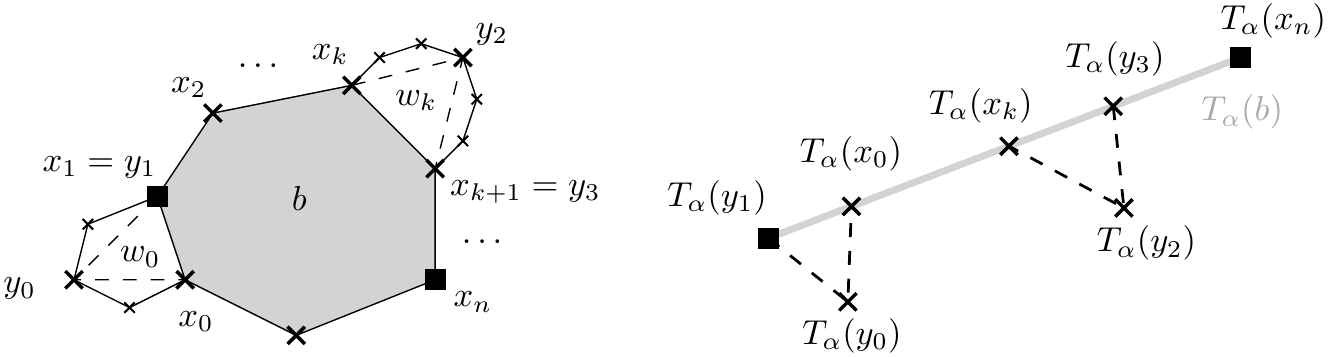}}
\caption{The assumption that the faces~$T_\alpha(w_0)$ and~$T_\alpha(w_k)$, where~$T_\alpha:=\cT+\alpha^2\cO$, have opposite orientation leads to a contradiction with the maximum principle for the projection~$v\mapsto \Re[\overline{\beta}T_\alpha(v)]$ onto the direction~$\beta\R\perp T_\alpha(b)$.}
\label{fig:orient}
\end{figure}

Let us show that~$\cT(w_0)$ and~$\cT(w_k)$ have the same orientation for each \mbox{$k\le m-1$}; together with a similar argument applied to~$\cT(w_m)$ and~$\cT(w_k)$, $k\ge m+1$, this easily implies that all white faces adjacent to~$b$ have the same orientation. Let~$y_0$ and~$y_2$ be some vertices of the faces~$w_0$ and~$w_k$, respectively, not adjacent to~$b$. Also, let us denote~$y_1:=x_1$ and~$y_3:=x_{k+1}$; see Fig.~\ref{fig:orient}. Similarly to the proof of the property (a) given above, we now apply the maximum principle provided by Lemma~\ref{lem:T-no-strict-max} to the function~$\Re[\overline{\beta}T_\alpha]$, where the direction~$\beta\in\mathbb{T}$ is chosen so that the line~$\beta\R$ is orthogonal to the segment~$T_\alpha(b)$. This implies that the points~$T_\alpha(y_0)$ and~$T_\alpha(y_2)$ must lie at different sides of the segment~$T_\alpha(b)$ since the projections of~$T_\alpha(y_1)$ and of~$T_\alpha(y_3)$ onto the direction~$\beta\R$ coincide. Therefore, $T_\alpha(w_0)$ and~$T_\alpha(w_k)$ have the same orientation and hence the same holds for~$\cT(w_0)$ and~$\cT(w_k)$.

It remains to prove the property (c). This easily follows from (a), (b) and from the argument principle: since all faces of the t-embedding~$\cT$ have the same orientation, a given point~$z$ in~$\C$ is covered exactly $\operatorname{ind}_P(z)$ times, where~$\mathrm{ind}_P(z)$ is the index of the boundary polygon~$P$ with respect to~$z$. Hence, either all points inside $P$ are covered exactly ones and those outside~$\partial\cT$ are not covered, or vice versa; the latter is clearly impossible.
\end{proof}

\subsection{Open questions} \label{sub:questions}
We conclude this paper by mentioning several open questions that naturally arise in the context of our work.
\begin{oquestion} Prove that each, sufficiently non-degenerate, weighted planar bipartite graph~$\G$ with a marked outer face admits a perfect t-embedding.
\end{oquestion}
It is worth noting that the simplest case~$\deg v_\mathrm{out}=4$ can be worked out by straightforward computations relying upon Theorem~\ref{thm:perfect-gauge}; see also~\cite[Theorem~2]{KLRR}. However, such a brute force analysis seems to be too difficult already if~$\G$ is a weighted hexagonal prism since the conditions $(\cT(v_k),\cO(v_k))\in\rH$ lead to biquadratic equations in the variables~$\cF^\tb(b_k)$,~$\cF^\tw(w_k)$, $k=1,\ldots,n$.

\begin{oquestion} Prove that perfect t-embeddings of a given weighted bipartite graph {(provided they exist)} are unique up to the transforms {of the gauge functions~$\cF^\tb,\cF^\tw$} described in Remark~\ref{rem:Lorentz-isometries}.
\end{oquestion}

Clearly, the existence and uniqueness of perfect t-embedding of big dimer graphs is a crucial ingredient for our approach to the analysis of the bipartite dimer model. Unfortunately, at the moment we are not aware of any result in this direction that goes beyond several particular cases.

Before formulating the next question, let us briefly describe why we expect that the assumption on the convergence of the graphs of origami maps to a \emph{Lorentz-minimal} surface~$(z,\vartheta(z))_{z\in \Omega}$ made in our Theorem~\ref{thm:main-GFF} is not artificial. It is not hard to see that if this surface is \emph{not} Lorentz-minimal, then the functions~$\psi_\frb$ and~$\psi_\frw$ introduced in Proposition~\ref{prop:hol-in-zeta} are not holomorphic in the conformal metric of the surface but rather satisfy the massive holomorphicity equations
\[
\opazeta\psi_\frb(\zeta)\ = m(\zeta)\overline{\psi_\frb(\zeta)},\qquad \opazeta\psi_\frw(\zeta)\ = \overline{m(\zeta)}\cdot \overline{\psi_\frb(\zeta)}
\]
with~$m(\zeta)\ne 0$. (E.g., see~\cite[Section~2.7]{chelkak-s-emb} for a relevant discussion under the additional assumption~$\vartheta\in\R$ specific to the Ising model setup.) In the theoretical physics language, in this setup the limit of height fluctuations should be described by the bosonization of free \emph{massive fermions}, which leads to the \emph{sine-Gordon theory} with a special value of the interaction parameter rather than to free bosons; e.g., see~\cite{bauerschmidt-webb} and references therein. Coming back to the dimer model language, {this means that we do not expect the height fluctuations to have a Gaussian structure in the small mesh size limit if~$m(\zeta)\ne 0$;} cf. concrete examples of such a behavior treated in~\cite{chhita}.

On the other hand, following Kenyon and Okounkov~\cite{kenyon-okounkov} it is widely believed that in many, if not all, reasonable setups for the dimer model considered on subgraphs of periodic grids, the Gaussian structure of height fluctuations \emph{do} appear in the scaling limit. This suggests that the appearance of Lorentz-minimal surfaces in the limit of the graphs of origami maps constructed out of (perfect) t-embeddings of big pieces of periodic grids should be rather a general phenomenon than a specific property of, e.g., the classical Aztec diamonds considered in~\cite{chelkak-ramassamy}. {This suggests the next question:}

\begin{oquestion} Find a proper notion of discrete Lorentz-minimal surfaces that correspond to the graphs of the origami maps obtained from (perfect) t-embeddings of big pieces of a given \emph{periodic} weighted bipartite graph~$\G$.
\end{oquestion}

Finally, it would be interesting to develop a general correspondence between the picture developed by Kenyon, Okounkov and Sheffield in~\cite{kenyon-okounkov,kenyon-okounkov-sheffield} on the one side and the context of perfect t-embeddings on the other. {As in the case of big homogeneous Aztec diamonds (see~\cite{chelkak-ramassamy}),} we expect that frozen regions are always collapsed to vertices of the polygonal domain~$\Omega_\xi$; this also seems to be the only scenario compatible with Theorem~\ref{thm:main-GFF}. Similarly, we conjecture that the \emph{gaseous bubbles} should collapse to points inside~$\Omega_\xi$ and the graphs of the origami maps should converge to \emph{Lorentz-minimal cusps} near these points. To a certain extent, this conjecture is supported by an explicit analysis of the dimer model on big cylinders obtained from the square grid and their perfect t-embeddings; see~\cite{russkikh-cusp} for details. Let us summarize the preceding discussion in the following vague question.
\begin{oquestion}
Find a precise correspondence between the classical phenomenology of frozen/liquid/gazeous zones appearing in the dimer model in polygonal domains on periodic grids on the one side, and the structure of perfect t-embeddings of these dimer graphs on the other.
\end{oquestion}
{Certainly,} the research program outlined {above} is very long and can become questionable in the unfortunate case that the existence/uniqueness of perfect t-embeddings does not hold in general. However, we believe that 
Theorem~\ref{thm:main-GFF} provides an interesting twist in the analysis of the bipartite dimer model by means of discrete complex analysis even if the range of its potential applications will turn out to be limited to a certain subclass of weighted bipartite graphs rather than to all of them.


\end{document}